\documentclass[8pt]{article}
\usepackage[a4paper, total={5.5in, 8in}]{geometry}
\usepackage[numbers]{natbib}
\usepackage{amsthm}
\usepackage{mathrsfs}
\usepackage[breaklinks,bookmarks=false]{hyperref}

\hypersetup{colorlinks, linkcolor=blue, citecolor=blue,
urlcolor=blue, plainpages=false, pdfwindowui=false,
pdfstartview={FitH}}

\usepackage{amsmath,amsfonts,amsthm,amssymb,mathtools,subcaption,enumitem,comment}
\usepackage{mathrsfs}
\usepackage[only,llbracket,rrbracket,llparenthesis,rrparenthesis]{stmaryrd} 
\usepackage[textsize=tiny,color=blue!10!white]{todonotes}
\usepackage{bm}
\usepackage{adjustbox}

\usepackage[medium,compact]{titlesec}
\titleformat*{\section}{\large\bfseries}
\titleformat*{\subsection}{\bfseries}

\numberwithin{equation}{section}

\newtheorem{theorem}{Theorem}[section]{\bfseries}{\it}
\newtheorem{proposition}[theorem]{Proposition}{\bfseries}{\it}
\newtheorem{lemma}[theorem]{Lemma}{\bfseries}{\it}
\newtheorem{corollary}[theorem]{Corollary}{\bfseries}{\it}
\newtheorem{definition}{Definition}[section]{\bfseries}{\it}
\newtheorem{example}{Example}[section]{\bfseries}{\rmfamily}
{\bfseries}{\it}
\theoremstyle{definition}
\newtheorem{remark}{Remark}[section]{\bfseries}{\rmfamily}
\newcommand{\proofbox}{\qed}

\newcommand{\tends}{\rightarrow}

\newcommand{\R}{\mathbb{R}}
\newcommand{\N}{\mathbb{N}}

\DeclareMathOperator{\supp}{supp}

\DeclareMathOperator*{\argmax}{arg\,max}
\DeclareMathOperator*{\conv}{conv}
\renewcommand{\dim}{d}
\newcommand{\norm}[1]{\lVert#1\rVert}

\title{Regularization of Stationary Second-order Mean Field Game Partial Differential Inclusions}
\author{Yohance A. P. Osborne\thanks{Department of Mathematical Sciences, Durham University, Stockton Road, DH1 3LE Durham, United Kingdom (\texttt{yohance.a.osborne@durham.ac.uk}).} 
	\and
	Iain Smears \thanks{Department of Mathematics, University College London, Gower
		Street, WC1E 6BT London, United Kingdom (\texttt{i.smears@ucl.ac.uk}).}}

\begin{document}

\maketitle



\begin{abstract}
Mean field Game (MFG) Partial Differential Inclusions (PDI) {are generalizations of} the system of Partial Differential Equations (PDE) of Lasry and Lions to situations where players in the game {may have possibly nonunique} optimal controls, and the resulting Hamiltonian is not required to be differentiable.
{We study second-order MFG PDI with convex, Lipschitz continuous, but possibly nondifferentiable Hamiltonians, and their approximation by systems of classical MFG PDE with regularized Hamiltonians.}
Under very broad conditions on {the problem} data, we show that{, up to subsequences, the solutions of the regularized problems converge to solutions of the MFG PDI. In particular, we show the convergence of the value functions in the $H^1$-norm and of the densities in $L^q$-norms.}
{Under stronger hypotheses on the problem data, we also show rates of convergence between the solutions of the original and regularized problems, without requiring any higher regularity of the solutions.
We give concrete examples that demonstrate the sharpness of several aspects of the analysis.}
\end{abstract}

\section{Introduction}\label{sec-introduction}
Mean Field Games (MFG), which {were} introduced by Lasry \& Lions \cite{lasry2006jeuxI,lasry2006jeuxII,lasry2007mean} and independently by Huang, Caines \& Malham\'e \cite{huang2006large}, {model the Nash equilibria}  {of} {dynamic} games of optimal control in which there are a large number of players.
{MFG encompass a wide range of models for handling various situations, yet for simplicity we concentrate here on a standard model of MFG~\cite{lasry2007mean}, where the players have identical stochastic dynamics with idiosyncratic noise which are independent of the density, and where the controls enter only in deterministic drift terms and in part of the running of the costs. 
	Under suitable assumptions, the} equilibria {are described by the solutions of a} system {of partial differential equations (PDE)} that consists of a Hamilton--Jacobi--Bellman (HJB) {equation, for} the value function of the optimal control problem faced by each player, and a Kolmogorov--Fokker--Planck (KFP) {equation, for} the distribution of {the} players {across the} state space of the game.
There is a {wide literature} on mean field games;  we refer the reader to~\cite{GueantLasryLions2003,GomesSaude2014,gomes2016regularity,achdou2020mean} for extensive reviews on the theory and applications of these systems.

{In many applications, the underlying optimal control problem can give rise to nonuniqueness of the optimal controls, when expressed in feedback form. This is closely related to the fact that in general, the Hamiltonian of the corresponding HJB equation can be nondifferentiable with respect to the derivatives of the value function.
	As a result, such situations fall outside the scope of many existing works on MFG, where differentiability of the Hamiltonian is often assumed.
	We stress that this is no mere technical issue, since the Nash equilibria can assume a more complicated structure in the more general case of nondifferentiable Hamiltonians and nonunique optimal controls. 
	This is because even for symmetric games, Nash equilibria are not necessarily symmetric, so there can be situations where it is necessary for the players at a given state to make distinct choices of optimal controls from each other in order to maintain an equilibrium.
	As a result of the possibly more complex structure of the Nash equilibria, it is natural to expect that there can be substantial qualitative differences between the cases of differentiable and nondifferentiable Hamiltonians, which is something that we will demonstrate in this work.}

{One approach to formulating MFG that allows for these more complicated Nash equilibria comes from the probabilistic literature~\cite{lacker2015mean}, where one explicitly seeks the player distribution over the product space of the states and controls, instead of the state-space only.
	Although conceptually instructive, we are however ultimately interested in the practical solution of MFG, which requires formulations of the problem that are well-suited for numerical discretization.
	It is therefore strongly preferable to avoid the cost of discretizing functions over the higher dimensional state-control space, which might be computationally intractable in many cases.
	This motivates the search for other formulations of the MFG that keep the same dimensionality of the original PDE system of Lasry and Lions.}
There are a handful of works on MFG systems with nondifferentiable Hamiltonians in the PDE literature.
In \cite{mazanti2019minimal} Mazanti and Santambrogio introduced a model of \emph{ deterministic} minimal time MFG with congestion, {where the Hamiltonian takes the form} $H(x,p,m)=\mathcal{K}(x,m)|p|$ for $(x,p,m)\in \Omega\times \mathbb{R}^d\times \mathcal{P}(\Omega)$.
Using the structure of the underlying deterministic optimal control problem, they show that $\nabla u\neq 0$ on the support of $m$, {thus avoiding the point of nondifferentiability of $H$.}
Extensions of this work were then considered in \cite{dweik2020sharp,arjmand2021nonsmooth,arjmand2022multipopulation} for MFG of minimal-time type in different modelling contexts.
{The possible nonuniqueness of classical solutions for nondifferentiable Hamiltonians was studied in~\cite{bardi2019non}, where again} the gradient of the value function avoids the points of nondifferentiability of the Hamiltonian.
Recently, {Ducasse, Mazanti \& Santambrogio} {considered in~\cite{ducasse2020second} a second-order} minimal-time MFG system with a Hamiltonian taking the form $H(p,m)=\mathcal{K}(m)|p|$, where the {advective term in the continuity equation for $m$} is required to satisfy the optimality conditions {related to the Pontryagin maximum principle}, but is otherwise possibly nonunique.

{In order to handle general nondifferentiable Hamiltonians,} we proposed in~\cite{osborne2022analysis,osborne2023finite} {to generalize the MFG system of~Lasry and Lions} to a system where the KFP equation is relaxed to a \emph{partial differential inclusion} (PDI) where {the coefficient of the advective term is an element of the subdifferential of the Hamiltonian.}
In particular, we showed that this notion of solution {includes that of~\cite{ducasse2020second} as a special case,} at least in terms of the dependence of the Hamiltonian with respect to the gradient of the value function, {see~\cite[Remark~3.2]{osborne2022analysis}.}
{The existence and uniqueness of solutions of the resulting MFG PDI system, along with the convergence of their numerical approximations, was then shown under suitable hypotheses for steady-state problems in~\cite{osborne2022analysis} and for time-dependent problems in~\cite{osborne2023finite}.}

{
	\subsection{MFG partial differential inclusions and heuristic derivation}
	Let us now consider a heuristic derivation of a MFG PDI system, which will help to motivate it and to give some intuition as to how it describes the Nash equilibria in the case of possibly nonunique optimal controls.}
For instance, we can consider a system of the form
\begin{subequations}\label{intro-sys-time-mfg-pdi-eg}
	\begin{align}
		-\partial_tu- \nu\Delta u+{H}(t,x,\nabla u)&={F}[m](t,x) \text{ }&&\text{ in } (0,T)\times\Omega,\label{intro-sys-time-hjb}
		\\
		\partial_tm-\nu\Delta m &\in \text{div}\left(m\partial_p {H}(t,x,\nabla u)\right)\text{ }&&\text{ in }(0,T)\times\Omega,\label{intro-sys-time-kfp}
	\end{align}
\end{subequations}
which models a generic single-population MFG in a state space $\Omega \subset \R^d$, $d\in\N$, over a finite time horizon $(0,T)$.
{Here $u:(0,T)\times\Omega\tends \R$ denotes the value function of the underlying optimal control problem, and $m:(0,T)\times\Omega\tends \R$ denotes the density of the player distribution.}
{
	To fix ideas, we shall restrict our attention to the case where $\Omega$ is a bounded domain, in which case the system~\eqref{intro-sys-time-mfg-pdi-eg} is accompanied by suitable boundary conditions for $m$ and $u$ on $\partial\Omega$. Moreover the system is coupled with suitable initial and final time conditions for $m$ and $u$ respectively.
	We refer the reader to~\cite{osborne2023finite} for more details on boundary and initial/final time conditions.}
The coupling term $F[m]$ denotes the component of the running cost that is allowed to depend on $m$, possibly in some nonlocal manner.
{The Hamiltonian $H:(0,T)\times\Omega\times\mathbb{R}^d\ni (t,x,p)\mapsto H(t,x,p)$ is a real-valued function that is convex with respect to the gradient variable $p\in \R^d$, and the set-valued map $\partial_pH:(0,T)\times\Omega\times\mathbb{R}^d\rightrightarrows \mathbb{R}^d$ denotes the partial subdifferential of $H$ with respect to $p$, which is defined by
	\begin{equation}
		\partial_pH(t,x,p)\coloneqq\left\{\tilde{b}\in\mathbb{R}^d:H(t,x,q)\geq H(t,x,p)+\tilde{b}\cdot (q-p)\quad\forall q\in\mathbb{R}^d\right\}.
	\end{equation}
	Note that~\eqref{intro-sys-time-mfg-pdi-eg} coincides with a classical MFG PDE system whenever $H$ is differentiable with respect to $p$, since then the subdifferential $\partial_p H$ is a singleton set containing only the partial derivative of $H$ with respect to $p$.
}

{The system~\eqref{intro-sys-time-mfg-pdi-eg} can be derived on a heuristic level as follows.}
{We suppose that the players' dynamics are given by a controlled stochastic process of the form
	\begin{equation}\label{SDE}
		\mathrm{d}X_t=-b(t,X_t,\alpha_t)\mathrm{d}t+\sqrt{2\nu}\mathrm{d}B_t,
	\end{equation} 
	where $b:[0,T]\times\Omega\times\mathcal{A}\to\mathbb{R}^d$ denotes the negative controlled drift, with each player choosing a control $\alpha_t$ in the control set $\mathcal{A}$ at time $t\in(0,T)$.
	Note that we use here the convention of a minus sign in the drift term in~\eqref{SDE} because it helps to simplify the notation below; this explains why we refer to~$b$ as the negative controlled drift.
	We will assume in the following that $\mathcal{A}$ is a compact metric space and the data $b$ and $f$ are uniformly continuous with respect to their arguments.
	Also here $\nu>0$ is a constant and $B_t$ denotes the standard $d$-dimensional Brownian motion. 
	Also, we assume that the process $X_t$ is either absorbed or reflected at the boundary $\partial\Omega$, which leads to either Dirichlet or Neumann boundary conditions for the MFG system, although the details are not necessary for the immediate discussion.}
{If we suppose that the objective functional of the optimal control problem is in separable form, i.e.\ the running cost increment per unit time is of the form $f(t,X_t,\alpha_t)+F[m](t,X_t)$, then the Hamiltonian $H$ is defined by}
\begin{equation}\label{intro-Ham}
	H(t,x,p)\coloneqq \sup_{\alpha\in\mathcal{A}}\left\{b(t,x,\alpha)\cdot p - f(t,x,\alpha)\right\}\quad\forall (t,x,p)\in (0,T)\times \Omega\times\mathbb{R}^d,
\end{equation}
{Then, at a Nash equilibrium, the value function $u$ solves the HJB equation~\eqref{intro-sys-time-hjb}, and we suppose momentarily that $u$ is sufficiently regular.}
Under the above hypotheses on $\mathcal{A}$ and the data $b$ and $f$, the supremum in~\eqref{intro-Ham} is achieved.
{Then, the players at a given state $x\in\Omega$ and time $t\in (0,T)$ make a choice of feedback control from} the maximizing set $\Lambda(t,x,\nabla u(t,x))$, where the set-valued map $\Lambda$ is defined by
\begin{equation}\label{eq:maximizing_controls}
	\begin{aligned}
		\Lambda(t,x,p)\coloneqq \argmax_{\alpha\in \mathcal{A}}\left[ b(t,x,\alpha)\cdot p - f(t,x,\alpha) \right] &&& \forall (t,x,p)\in (0,T)\times \Omega\times\R^\dim.
	\end{aligned}
\end{equation}
As explained above, the maximizing set~\eqref{eq:maximizing_controls} is not always a singleton,  so it is possible for the players at $(t,x)$ to make distinct choices of controls from each other.
{To model this, we suppose that the players' positions and choices of control can be described} by a flow of measures $\{M_t\}_{t\in [0,T]}$ on $\Omega\times\mathcal{A}$ that satisfies the KFP equation
\begin{equation}\label{weak-kfp-occ-m-eqn}
	\int_0^T\int_{\Omega\times\mathcal{A}}\big(-\partial_t\phi(t,x)-\nu\Delta\phi(t,x)+b(t,x,\alpha)\cdot\nabla \phi(t,x)\big)M_t(\mathrm{d}x,\mathrm{d}\alpha)\mathrm{d}t=0,
\end{equation}
{for all smooth compactly-supported test functions $\phi$ on $(0,T)\times \Omega$.}
In order to connect the flow of measures $\{M_t\}_{t\in (0,T)}$ with {the players' choices of controls}, we consider the formal disintegration of measures
\begin{equation}\label{rho-disint-eqn}
	M_t(\mathrm{d}x,\mathrm{d}\alpha) = \rho_{t,x}(\mathrm{d}\alpha)m_t(\mathrm{d}x) \quad {\forall t\in[0,T]},
\end{equation}
where $m_t(\mathrm{d}x) = \int_{\mathcal{A}}M_t(\mathrm{d}x,\mathrm{d}\alpha)$ is the measure on $\Omega$ obtained by projection of $M_t$ on $\Omega$, and where $\rho_{t,x}\in \mathcal{P}(\mathcal{A})$ is the probability measure on $\mathcal{A}$ {representing the choices of controls for the players at $(t,x)$.}
{At a Nash equilibrium, the} players at $(t,x)$ choose controls from the feedback control set $\Lambda(t,x,\nabla u(t,x))$, and thus {we require that}
\begin{equation}\label{mfg-pdi-hypo}
	\begin{aligned}
		\text{supp }\rho_{t,x}\subset \Lambda(t,x,\nabla u(t,x)) &&&\forall (t,x)\in (0,T)\times\Omega,
	\end{aligned}
\end{equation}
where $\supp\rho_{t,x}$ denotes the support of the measure $\rho_{t,x}$.

{
	We now show how to obtain the differential inclusion~\eqref{intro-sys-time-kfp}. If the flow of measures $\{m_t\}_{t}$ is sufficiently regular to admit a sufficiently smooth density, i.e.\ if~$m_t(\mathrm{d}x)=m(t,x)\mathrm{d}x$ with $\mathrm{d}x$ denoting the Lebesgue measure, then by using~\eqref{rho-disint-eqn}, Fubini's theorem, and the fact that $\rho_{t,x}$ is a probability measure on $\mathcal{A}$, we obtain from~\eqref{weak-kfp-occ-m-eqn} the KFP equation for the density $m$ of the form
	\begin{align}
		\int_0^T\int_{\Omega}\big(-\partial_t\phi-\nu\Delta\phi+\tilde{b}_*\cdot\nabla \phi\big)m \,\mathrm{d}x\mathrm{d}t= 0 &&&\forall \phi \in C^\infty_0((0,T)\times \Omega)\label{weak-kfp-occ-m-eqn-alt},
		\\
		\tilde{b}_*(t,x)\coloneqq \int_{\mathcal{A}}b(t,x,\alpha)\rho_{t,x}(\mathrm{d}\alpha) &&& \forall (t,x)\in (0,T)\times\Omega.
	\end{align}
	In other words, the density $m$ satisfies a KFP equation with a (negative) drift term $\tilde{b}_*$ which represents a weighted average of the (negative) drifts that result from the players' choices.}
Recalling the hypothesis that $\mathcal{A}$ is a compact metric space and the data $b$ and $f$ are continuous, the optimality condition~\eqref{mfg-pdi-hypo} and the fact that $\rho_{t,x}\in \mathcal{P}(\mathcal{A})$ then imply that
\begin{equation}\label{b-vec-inclusion-1}
	\begin{aligned}
		\tilde{b}_*(t,x)\in \text{conv}\{b(t,x,\alpha):\alpha\in \Lambda(t,x,\nabla u(t,x))\} &&&\forall (t,x)\in (0,T)\times\Omega.
	\end{aligned}
\end{equation}
It is known from Convex Analysis that the convex hull of the set of optimal (negative) drifts is precisely the subdifferential of the Hamiltonian with respect to $p$:
\begin{equation}\label{intro-convv-formula}
	\conv\{b(t,x,\alpha):\alpha\in \Lambda(t,x,p)\} = \partial_pH(t,x,p) \quad\forall (t,x,p)\in (0,T)\times\Omega\times \R^\dim,
\end{equation}
see for instance~\cite[Lemma~3.4.1]{osborne2024thesis} for an elementary proof of~\eqref{intro-convv-formula} involving only Ca\-ra\-theo\-do\-ry's theorem and the hyperplane separation theorem.
Therefore~\eqref{b-vec-inclusion-1} is equivalent to $\tilde{b}_*(t,x)\in \partial_p H(t,x,\nabla u(t,x))$ for all $(t,x)$.
If $m$ is regular enough to recast the KFP equation~\eqref{weak-kfp-occ-m-eqn-alt} into strong form, we then obtain {the differential inclusion~\eqref{intro-sys-time-kfp} as initially claimed.
	Although the arguments in the paragraphs above are only heuristic, they nonetheless help to give some insight into the meaning of the MFG PDI, and its connection to the structure of the underlying Nash equilibria.
}

\subsection{Contributions of this work}

{In addition to the heuristic derivation above, we show in this work how the PDI system generalizes the original PDE system to the case of nondifferentiable Hamiltonians.
	More precisely, we show that the MFG PDI system can also be understood as the limit of sequences of the MFG PDE systems with regularized Hamiltonians.
	The idea of approximating the MFG PDI by regularized PDE systems is also motivated by numerical computations, since it paves the way to approximating the solutions of the resulting discretized PDI in~\cite{osborne2022analysis,osborne2023finite} by existing solvers for nonlinear equations.}
To be more concrete, let us consider for simplicity {as a model problem} the stationary MFG PDI system
\begin{equation}\label{mfg-pdi-sys}
	\begin{aligned}
		- \nu\Delta u+H(x,\nabla u)&=F[m] \text{ }&&\text{ in }\Omega,
		\\
		-\nu\Delta m - G(x)&\in \text{div}\left(m\partial_p H(x,\nabla u)\right)\text{ }&&\text{ in }\Omega,
		\\
		u=0\quad\text{and}\quad m&=0,\quad &&\text{on }\partial\Omega,
	\end{aligned}
\end{equation}
{where $\nu>0$ is a constant, where the Hamiltonian $H$ is convex and Lipschitz continuous w.r.t.\ $p$, but not necessarily differentiable. We specify the assumptions on the coupling term $F$ and the source term $G$ in~Section~\ref{sec-notation} below.}
Note that the case of Dirichlet boundary conditions in~\eqref{mfg-pdi-sys} corresponds to models of infinite horizon control, where the players exit the game when they reach the boundary $\Omega$. We assume that the Dirichlet conditions on $u$ and $m$ are homogeneous for simplicity, yet note that the inhomogeneous case is usually reduced to the homogeneous one by suitable transformation of the problem. The source/sink term $G$ represents the source/sink of new players into the game.

We shall study the approximation of~\eqref{mfg-pdi-sys} by regularized problems of the form
\begin{equation}\label{eq:regularized_MFG}
	\begin{aligned}
		- \nu\Delta {u}_{\lambda}+{H}_{\lambda}(x,\nabla {u}_{\lambda})&={F}[{m}_{\lambda}] \text{ }&&\text{ in }\Omega,
		\\
		-\nu\Delta {m}_{\lambda} - \text{div}\left({m}_{\lambda}\frac{\partial {H}_{\lambda}}{\partial p}(x,\nabla {u}_{\lambda})\right) &= {G}(x)\text{ }&&\text{ in }\Omega,
		\\
		u_{\lambda}=0 \text{ }\text{ and }\text{ }m_{\lambda}&=0 &&\text{ on }\partial \Omega,
	\end{aligned}
\end{equation} 
where $\{H_{\lambda}\}_{\lambda\in(0,1]}$ is a family of {regularized Hamiltonians that are convex and $C^1$ w.r.t $p$ and} that converges uniformly to $H$ as the regularization parameter $\lambda \tends 0$.
{Precise assumptions on the regularizations are given in Section~\ref{sec:regularization} below, which include many well-known choices such as Moreau--Yosida regularization {and mollification}, c.f.~\eqref{H-moreau-yosida-reg-defn} and~\eqref{H-mollifcation} below.}
{Under some mild conditions on the data $F$ and $G$, the existence of weak solutions in $H^1_0$ of~\eqref{mfg-pdi-sys} and of~\eqref{eq:regularized_MFG} is already known, c.f.~\cite{osborne2022analysis,osborne2024thesis}, with the uniqueness result for strictly monotone couplings from Lasry and Lions~\cite{lasry2007mean} extending also to the PDI case.}
{Note that we choose to focus our analysis on the setting of weak solutions in $H^1_0$, since it is the most relevant for numerical methods, c.f.~\cite{osborne2022analysis,osborne2024near}. 
	One could also consider other analytical settings, however note that solutions of the MFG PDI~\eqref{mfg-pdi-sys} can have rather limited regularity in general as a result of the nondifferentiability of $H$, see~\cite[Section~3.3]{osborne2022analysis}.}

Our first main result, in Theorem~\ref{theorem-convergence-moreau-yosida} below, shows that solutions of the regularized problems~\eqref{eq:regularized_MFG} converge, up to subsequences and in suitable norms, to a weak solution of the PDI system~\eqref{mfg-pdi-sys}. {This makes precise the sense in which MFG PDI generalize MFG PDE when relaxing the differentiability assumption on the Hamiltonian, and one can also} consider this result as an alternative proof of the existence of weak solutions to MFG PDI.
We further demonstrate the sharpness of the conclusions in Theorem~\ref{theorem-convergence-moreau-yosida} through various examples in Section~\ref{sec-examples}.

Our second main result, in Theorem~\ref{theorem-rate-of-convergence-moreau-yosida} below, shows that if stronger quantitative assumptions are placed on the data, including a strong monotonicity condition on $F$,  {thus ensuring that the MFG PDI has a unique solution}, then we can prove a rate of convergence for the differences $u-u_{\lambda}$ and $m-m_{\lambda}$ in terms of error between $H$ and $H_{\lambda}$ of the form
\begin{align}\label{rates-of-convergence}
	\norm{m-m_{\lambda}}_{\mathcal{X}} + \norm{u-u_{\lambda}}_{H^1(\Omega)} \leq C \omega(\lambda)^{\frac{1}{2}},
\end{align}
for all $\lambda$ sufficiently small, where the constant~$C$ is independent of $\lambda$, where $\omega(\lambda)$ is a maximum-norm {upper} bound for the difference in the Hamiltonians, c.f.~\eqref{H-bounds:reg-unif-approx_1} below, and where $\mathcal{X}$ is the space on which $F$ is defined.
In particular, the bound~\eqref{rates-of-convergence} does not assume any higher regularity assumption on the solution~$(u,m)$ of~\eqref{mfg-pdi-sys}.
The bound~\eqref{rates-of-convergence} is particularly relevant for numerical computations since one must balance the errors that stem from different sources, including discretization, regularization, iteration, etc. {The application of~\eqref{rates-of-convergence} to numerical methods will be the subject of future work.}

This paper is organized as follows.
In Section~\ref{sec-notation} we establish notation and introduce hypotheses on the model data. 
{Section~\ref{sec-weakform-def-and-reg} introduces the MFG PDI and its regularization.}
{We state the main results in Section~\ref{sec-main-results}.}
Preliminary results are established in Section~\ref{sec-analysis-prelim-results}, which we use in Section~\ref{sec-pfs-of-main-result-1} to prove Theorem~\ref{theorem-convergence-moreau-yosida} and Corollary~\ref{corollary-strong-H1-conv-density-1}. 
We prove Theorem~\ref{theorem-rate-of-convergence-moreau-yosida} on convergence rates in Section~\ref{sec-pfs-of-main-result-2}. 
Examples that {demonstrate the sharpness of our conclusions are given} in Section~\ref{sec-examples}.

\section{Setting and Notation}\label{sec-notation}
We denote $\mathbb{N}\coloneqq \{1,2,3,\cdots\}$.
For a Lebesgue measurable set $A \subset \mathbb{R}^d$, $d\in\mathbb{N}$, let $\lVert \cdot \rVert_{L^2(A)}$ denote the standard $L^2$-norm for scalar- and vector-valued functions on $A$. 
{For $d\in\mathbb{N}$}, the $d$-dimensional open ball of radius $r$ and centre $x_0\in\mathbb{R}^d$ is denoted by $B_r(x_0)$.
Let $\Omega$ be a bounded, open connected subset of $\mathbb{R}^d$ with Lipschitz boundary $\partial \Omega$. Let the diffusion $\nu>0$ be constant, and let $G\in H^{-1}(\Omega)$ be given. Let $(\mathcal{X},\|\cdot\|_{\mathcal{X}})$ be a real Banach space such that $H_0^1(\Omega)$ is embedded continuously and compactly in $\mathcal{X}$.
We suppose that $F:\mathcal{X}\to H^{-1}(\Omega)$ is a continuous operator that satisfies
\begin{equation}\label{F-linear-growth}
	\|F[w]\|_{H^{-1}(\Omega)}\leq C_F\left(\|w\|_{\mathcal{X}}+1\right)\quad\forall w\in \mathcal{X},
\end{equation}
where $C_F\geq 0$ is a constant.
We introduce the following set of hypotheses on the source term $G$ and coupling operator $F$ that will be occasionally used in the subsequent analysis. 
\begin{enumerate}[label={(H\arabic*)}]
	\item\label{H:G-postive}
	$G\in H^{-1}(\Omega)$ is nonnegative in the sense of distributions, i.e.\
	$\langle G, \phi\rangle_{H^{-1}\times H_0^1} \geq 0$ for all functions $\phi\in H_0^1(\Omega)$ that are nonnegative a.e.\ in $\Omega$.
	
	\item\label{H:F-strict-mono}
	$F$ is {strictly monotone} on $H_0^1(\Omega)$, which is to say 
	\begin{equation}\label{eq:F_monotone}
		\langle F[m_1]-F[m_2],m_1-m_2\rangle_{H^{-1}\times H_0^1} \leq 0\Longrightarrow m_1=m_2,
	\end{equation} 
	whenever $m_1,\,m_2\in H_0^1(\Omega)$.
	
	\item \label{H:F-strong-mono-2}
	$F$ is strongly monotone on $H^1_0(\Omega)$ w.r.t {to the norm $\norm{\cdot}_{\mathcal{X}}$}, in the sense that there exists a constant $c_{F}>0$ such that 
	\begin{equation}\label{strong-mono-F}
		\langle {F}[m_1]-{F}[m_2], m_1-m_2\rangle_{H^{-1}\times H_0^1}\geq c_{F}\|m_1-m_2\|_{\mathcal{X}}^{2},
	\end{equation}for all $m_1,\,m_2\in H_0^1(\Omega)$.
	
	\item \label{H:F-Lipschitz-continuous}
	$F$ is Lipschitz continuous: 
	\begin{equation}
		\|F[m_1]-F[m_2]\|_{H^{-1}(\Omega)}\leq L_F\|m_1-m_2\|_{\mathcal{X}}\quad\forall m_1,m_2\in  \mathcal{X}.
	\end{equation}
	for some constant $L_F\geq 0$.
\end{enumerate}
{Note that in the subsequent results, we do not always require all of the hypotheses~\ref{H:G-postive}, \ref{H:F-strict-mono}, \ref{H:F-strong-mono-2}, and~\ref{H:F-Lipschitz-continuous}; therefore we will specify which hypotheses are needed in each case.}
{We also stress that in hypotheses~\ref{H:F-strict-mono} and~\ref{H:F-strong-mono-2}, the monotonicity condition is only required for arguments of $F$ in the smaller space $H^1_0(\Omega)$, and not the whole space $\mathcal{X}$, so the duality pairings in~\eqref{eq:F_monotone} and~\eqref{strong-mono-F} are well-defined.
	We refer the reader to~\cite[Section~2]{osborne2022analysis} for} some concrete examples of coupling operators that satisfy these conditions {in the case $\mathcal{X}=L^2(\Omega)$.}

{Motivated by the underlying optimal control problem, we suppose that the Hamiltonian~$H$ appearing in~\eqref{mfg-pdi-sys} is defined by
	\begin{equation}\label{Hamiltonian}
		H(x,p)\coloneqq \sup_{\alpha\in\mathcal{A}}\left\{b(x,\alpha)\cdot p-f(x,\alpha)\right\},\quad \forall (x,p)\in\overline{\Omega}\times\mathbb{R}^d,
	\end{equation}
	where} $b$ and $f$ are uniformly continuous on $\overline{\Omega}\times \mathcal{A}$ with $\mathcal{A}$ a compact metric space.
It follows that the Hamiltonian $H$ is Lipschitz continuous and satisfies 
\begin{equation}\label{bounds:lipschitz}
	|H(x,p)-H(x,q)|\leq L_H|p-q|\quad\forall(x,p,q)\in\overline{\Omega}\times\mathbb{R}^d\times\mathbb{R}^d,
\end{equation}
with {some constant $L_H$}, {for instance one can take $L_H\coloneqq\|b\|_{C(\overline{\Omega}\times\mathcal{A};\mathbb{R}^d)}$.}
We deduce from~\eqref{bounds:lipschitz} that there exists a constant $C_H\geq 0$ such that
\begin{equation}
	|H(x,p)|\leq C_H(|p|+1)\quad\forall (x,p)\in \overline{\Omega}\times\mathbb{R}^d.
\end{equation}
It is clear from~\eqref{bounds:lipschitz} that the mapping $v\mapsto H(\cdot,\nabla v)$ is Lipschitz continuous from $H^1(\Omega)$ into $L^2(\Omega)$. We will often abbreviate this composition by writing instead $H[\nabla v]\coloneqq H(\cdot,\nabla v)$ a.e.\ in $\Omega$.

Given arbitrary sets $A$ and $B$, an operator $\mathcal{M}$ that maps each point $x\in A$ to a \emph{subset} of $B$ is called a \emph{set-valued map from $A$ to $B$}, and we write $\mathcal{M}:A\rightrightarrows B$. For the Hamiltonian given by~\eqref{Hamiltonian} its point-wise  subdifferential with respect to $p$ is the set-valued map $ \partial_p H\colon\Omega\times\mathbb{R}^d\rightrightarrows\mathbb{R}^d$ defined by 
\begin{equation}\label{subdifferential}
	\partial_p H(x,p)\coloneqq\left\{\tilde{b}\in\mathbb{R}^d:H(x,q)\geq H(x,p)+\tilde{b}\cdot(q-p),\quad\forall q\in\mathbb{R}^d\right\}.
\end{equation}
Note that $\partial_p H(x,p)$ is nonempty for all $x\in\Omega$ and $p\in\mathbb{R}^d$ because $H$ is real-valued and convex in $p$ for each fixed $x\in\Omega$. 
Note also that the subdifferential $\partial_p H$ is uniformly bounded since~\eqref{bounds:lipschitz} implies that for all $(x,p)\in\Omega\times \mathbb{R}^d$, the set $\partial_p H(x,p)$ is contained in the closed ball of radius $L_H=\lVert b \rVert_{C(\overline{\Omega}\times\mathcal{A};\mathbb{R}^d)}$ centred at the origin.

{We now define a set-valued mapping for measurable selections of the subdifferential composed with the gradients of weakly differentiable functions.}
Given a function $v\in W^{1,1}(\Omega)$, we say that a real-valued vector field $\tilde{b}:\Omega\to\mathbb{R}^d$ is a \emph{measurable selection of $\partial_pH(\cdot,\nabla v)$} if $\tilde{b}$ is Lebesgue measurable and $\tilde{b}(x)\in \partial_p H(x,\nabla v(x))$ for a.e.\ $x\in \Omega.$
The uniform boundedness of the subdifferential sets implies that any measurable selection $\tilde{b}$ of $\partial_p H(\cdot,\nabla v)$ must belong to $L^\infty(\Omega;\mathbb{R}^d)$. 
Therefore, given $v\in W^{1,1}(\Omega)$ we can consider the set of all measurable selections of $\partial_p H(\cdot,\nabla v)$, which gives rise to a set-valued mapping from $W^{1,1}(\Omega)$ to subsets of $L^\infty(\Omega;\mathbb{R}^d)$.
\begin{definition}[\cite{osborne2022analysis}]\label{Def1}
	Let $H$ be the function given by~\eqref{Hamiltonian}. We define the set-valued map $D_pH\colon W^{1,1}(\Omega)\rightrightarrows L^{\infty}(\Omega;\mathbb{R}^d)$ by
	$$D_pH[v]\coloneqq \left\{\tilde{b}\in L^{\infty}(\Omega;\mathbb{R}^d):\tilde{b}(x)\in\partial_pH(x,\nabla v(x)) \text{ \emph{for a.e.\ }}x\in\Omega\right\}.$$
\end{definition}
In \cite[Lemma 4.3]{osborne2022analysis} it was shown that $D_pH[v]$ is a nonempty subset of $L^{\infty}(\Omega;\mathbb{R}^d)$ for each $v$ in $ W^{1,1}(\Omega)$. 

\paragraph{Notation for inequalities}
{In order to avoid the proliferation of generic constants,} in the following, we write $a\lesssim b$ for real numbers $a$ and $b$ if $a\leq C b$ for some constant $C$ that may depend on the problem data appearing below, but is otherwise independent of the regularization parameter $\lambda$ {appearing in~\eqref{eq:regularized_MFG}.}
Throughout this work we will regularly specify the particular dependencies of the hidden constants.

\section{Weak formulation of MFG PDI and of its regularizations}\label{sec-weakform-def-and-reg}

\subsection{Weak formulation of the MFG PDI}
{We recall the definition of a weak solution of the MFG PDI~\eqref{mfg-pdi-sys} that was introduced in~\cite{osborne2022analysis}.}
\begin{definition}[Weak Solution of MFG PDI~\eqref{mfg-pdi-sys}]\label{weakdef}
	A pair $(u,m)\in H_0^1(\Omega)\times H_0^1(\Omega)$ is a weak solution of~\eqref{mfg-pdi-sys} if there exists a $\tilde{b}_*\in D_pH[u]$ such that
	\begin{subequations}\label{weakform}
		\begin{align}
			\int_{\Omega}\nu\nabla u\cdot\nabla \psi+H[\nabla u]\psi \mathrm{d}x &=\langle F[m],\psi\rangle_{H^{-1}\times H_0^1} &&\forall \psi\in H^1_0(\Omega), \label{weakform1}\\ 
			\int_{\Omega}\nu\nabla m\cdot\nabla \phi+m\tilde{b}_*\cdot\nabla \phi \mathrm{d}x &=\langle G,\phi\rangle_{H^{-1}\times H_0^1} &&\forall \phi\in H^1_0(\Omega).\label{weakform2} 
		\end{align} 
	\end{subequations}
\end{definition}
\begin{remark}
	The above definition of weak solution only requires the existence of a suitable transport vector field $\tilde{b}_*\in D_pH[u]$, but its uniqueness is not required in general.
	{We refer the reader to~\cite[Section~3.3]{osborne2024thesis} for several examples showing that $\tilde{b}_*$ might be unique in some cases, but not in others.}
\end{remark}

Under the above hypotheses on the data, we have the following result on the existence and uniqueness of solutions.
\begin{theorem}[Existence and uniqueness of weak solutions~\cite{osborne2022analysis}]\label{theorem-existence-uniqueness-mfg-pdi}
	Let $\nu>0$ be constant, and let $G\in H^{-1}(\Omega)$. Let $H$ be the function given by~\eqref{Hamiltonian}, and let $F:\mathcal{X}\to H^{-1}(\Omega)$ be a continuous operator {satisfying~\eqref{F-linear-growth}.
		Then,} there exists a weak solution $(u,m)\in H_0^1(\Omega)\times H_0^1(\Omega)$ of~\eqref{mfg-pdi-sys} in the sense of Definition~\ref{weakdef}.
	{In addition, if} $G$ satisfies~\ref{H:G-postive} and if $F$ satisfies~\ref{H:F-strict-mono} then there is at most one weak solution of~\eqref{mfg-pdi-sys} in the sense of Definition~\ref{weakdef}.
\end{theorem}

The proof of {Theorem~\ref{theorem-existence-uniqueness-mfg-pdi} follows the same argument as in~\cite[Theorem~3.2.1]{osborne2024thesis} {(see also \cite{osborne2022analysis,osborne2024erratum})}, and is based on an application of Kakutani's fixed point theorem. For completeness, we include the proof of Theorem~\ref{theorem-existence-uniqueness-mfg-pdi} in Section~\ref{sec:proof_existence}.}

{
	\begin{example}[Nonuniqueness of solutions for nonmonotone $F$]\label{ex:nonuniqueness}
		In Theorem~\ref{theorem-existence-uniqueness-mfg-pdi}, uniqueness of solutions is shown under the additional conditions~\ref{H:G-postive} and~\ref{H:F-strict-mono}, which respectively require that $G$ is nonnegative and that $F$ is strictly monotone. 
		If these assumptions are relaxed, then in general uniqueness of solutions may fail, even for differentiable Hamiltonians; see for instance~\cite{bardi2019non} for some examples of nonuniqueness of solutions for some time-dependent MFG systems. 
		Here we give a short original example of nonuniqueness of solutions for the steady-state MFG system with homogeneous Dirichlet boundary conditions, when $F$ is not required to satisfy~\ref{H:F-strict-mono}, for a Hamiltonian $H$ that is $C^1$ with respect to $p$.
		For this example, suppose that $\Omega=(-1,1)\subset\mathbb{R}$, and let the Hamiltonian $H\colon \Omega\times \R\rightarrow \R$ be any function satisfying the condition that $H(x,p)=\frac{1}{2} p^2$ for all $p\in [-1,1]$ and all $x\in\Omega$; in particular we can take $H$ to be convex, globally Lipschitz, and $C^1$ with respect to $p$.
		Let the pairs of functions $(u_i,m_i)$, $i\in\{1,2\}$, be defined by
		\begin{subequations}
			\begin{align} 
				&u_{1}(x)\coloneqq \frac{1-x^2}{2}, \qquad m_{1}(x)\coloneqq 1 - e^{\frac{x^2-1}{2}}\quad\forall x\in [-1,1] \\
				&u_{2}(x)\coloneqq \frac{x^2-1}{2},\qquad m_{2}(x)\coloneqq e^{\frac{1-x^2}{2}}-1\quad\forall x\in [-1,1].
			\end{align} 
		\end{subequations}
		Note then that the $u_i$ and $m_i$, $i\in\{1,2\}$, are smooth functions in $\Omega$ satisfying homogeneous Dirichlet conditions on $\partial\Omega$.
		Note also that $|u_i^\prime (x)|=|x|\leq 1$ for all $x\in \Omega$, for each $i\in \{1,2\}$, so that $H(x, u_i^\prime(x))=\frac{|x|^2}{2}$ and $\frac{\partial H}{\partial p}(x,u_i^\prime(x))=u_i^\prime(x)$ for all $x\in\Omega$, for $i\in\{1,2\}$.
		Then, let $G\equiv 1$ in $\Omega$, and note that $G$ satisfies~\ref{H:G-postive}. Define $F$ by
		\begin{equation}\label{F-non-mono-example}
			F[m](x) \coloneqq \left(\frac{x^2}{2}+1\right)\frac{\norm{m-m_2}_{L^2(\Omega)}}{\norm{m_1-m_2}_{L^2(\Omega)}}+\left(\frac{x^2}{2}-1\right)\frac{\norm{m-m_1}_{L^2(\Omega)}}{\norm{m_2-m_1}_{L^2(\Omega)}}.
		\end{equation}
		Observe that $F$ is Lipschitz continuous from $\mathcal{X}=L^2(\Omega)$ to $H^{-1}(\Omega)$ and thus satisfies~\eqref{F-linear-growth}.
		However, it is clear that~$F$ does not satisfy the strict monotonicity condition~\ref{H:F-strict-mono}.
		It is then straightforward to check that $(u_1,m_1)$ and $(u_2,m_2)$ are both distinct classical solutions of the MFG system $-u^{\prime\prime} + H(x,u^\prime) = F[m]$ and $-m^{\prime\prime} - \left(m \frac{\partial H}{\partial p}(x,u^\prime)\right)^\prime = G$ in $\Omega$ along with homogeneous Dirichlet conditions for both $u$ and $m$ on $\partial\Omega$.
		This illustrates the possibility of nonuniqueness of solutions in the general case when~\ref{H:F-strict-mono} is not satisfied. 
\end{example}	}

\subsection{Regularized problems}

\subsubsection{A family of regularized Hamiltonians}\label{sec:regularization}

In order to analyse the regularized problems~\eqref{eq:regularized_MFG} in a unified way for a variety of different choices of regularizations, we consider a family of regularized Hamiltonians $\{H_{\lambda}\}_{\lambda\in(0,1]}$ that satisfies the following hypotheses.
\begin{enumerate}[label={(H\arabic*)},resume]
	\item $\{H_{\lambda}\}_{\lambda\in(0,1]}$ is a family of real-valued functions on $\overline{\Omega}\times \R^\dim$ for which there exists a continuous function $\omega:{[0,1]}\to [0,\infty)$ satisfying $\omega(0)=0$ such that, for each $\lambda\in (0,1]$:
	\begin{itemize}
		\item  $H_{\lambda}$ is continuous on $\overline{\Omega}\times \R^\dim$ and
		\begin{equation}\label{H-bounds:lipschitz-lam}
			\begin{aligned}
				|H_{\lambda}(x,p)-H_{\lambda}(x,q)|\leq L_H\lvert p-q\rvert &&&\forall(x,p,q)\in\overline{\Omega}\times\mathbb{R}^d\times\mathbb{R}^d,
			\end{aligned}
		\end{equation}
		\item for each $x\in\overline{\Omega}$, the map $\mathbb{R}^d\ni p\mapsto H_{\lambda}(x,p)$ is convex; 
		\item the partial derivative $\frac{\partial H_{\lambda}}{\partial p}:\overline{\Omega}\times\mathbb{R}^d\to\mathbb{R}^d$ exists and is continuous;
		\item $H_{\lambda}$ satisfies the following inequality
		\begin{equation}
			\begin{aligned}
				|H_{\lambda}(x,p) - H(x,p)|\leq \omega(\lambda)&& &\forall (x,p)\in \overline{\Omega}\times\mathbb{R}^d \label{H-bounds:reg-unif-approx_1}.
			\end{aligned}
		\end{equation} 
	\end{itemize}\label{H:reg-family} 
\end{enumerate}
{Note that in~\eqref{H-bounds:lipschitz-lam} we are supposing that there is a uniform Lipschitz constant $L_H$ for the original Hamiltonian $H$ and the family of regularizations $\{H_\lambda\}_{\lambda\in (0,1]}$.}
Furthermore, the bounds~\eqref{H-bounds:lipschitz-lam} and~\eqref{H-bounds:reg-unif-approx_1} imply that, for all $\lambda\in (0,1]${,
	\begin{equation}\label{H-bounds:linear-growth-lam}
		|H_{\lambda}(x,p)|\leq |H(x,0)|+\sup_{\sigma\in [0,1]}\omega(\sigma) + L_H|p|
		\leq \widetilde{C}_H\left(|p|+1\right)\quad \forall (x,p)\in \overline{\Omega}\times\mathbb{R}^d,
	\end{equation}
	f}or some constant $\widetilde{C}_H\geq 0$ independent of $\lambda\in (0,1]$.
{Note also that the restriction of $\lambda$ to the interval $(0,1]$ is not essential, e.g.\ it can be replaced by some more general bounded sets for which $0$ is a limit point.}

\subsubsection{Examples of families of regularized Hamiltonians}
{In this subsection, we illustrate several possible choices of regularizations that satisfy~\ref{H:reg-family}.}

\begin{example}
	{As a first example, we} consider the Moreau--Yosida regularization of the Hamiltonian~\eqref{Hamiltonian} with respect to $p$. For each $\lambda\in (0,1]$, let $H_{\lambda}:\overline{\Omega}\times\mathbb{R}^d\to\mathbb{R}$ denote the Moreau--Yosida regularization of $H$ w.r.t.\ $p$ defined by
	\begin{equation}\label{H-moreau-yosida-reg-defn}
		H_{\lambda}(x,p)\coloneqq \inf_{q\in\mathbb{R}^d}\left\{H(x,q) + \frac{1}{2\lambda }|q-p|^{2}\right\}.
	\end{equation}
	The following lemma shows that Moreau--Yosida regularization leads to a family of Hamiltonians that satisfies~\ref{H:reg-family}. 
\end{example}

\begin{lemma}[Moreau--Yosida regularization]\label{lemma-moreau-yosida-approx-properties}
	For each $\lambda\in (0,1]$, let $H_{\lambda}$ be defined by~\eqref{H-moreau-yosida-reg-defn}. 
	Then the family $\{H_{\lambda}\}_{\lambda\in (0,1]}$ satisfies~\ref{H:reg-family}.
	In particular, the bound~\eqref{H-bounds:lipschitz-lam} holds with same Lipschitz constant as in~\eqref{bounds:lipschitz} and~\eqref{H-bounds:reg-unif-approx_1} holds with $\omega(\lambda)=\frac{L_H^2\lambda}{2}$, i.e.\
	\begin{equation}\label{moreau-yosida-bounds:reg-unif-approx_1}
		\sup_{(x,p)\in \overline{\Omega}\times\mathbb{R}^d}|H_{\lambda}(x,p) - H(x,p)|\leq {\frac{L_H^2\lambda }{2}}.
	\end{equation}
\end{lemma}
\begin{proof}
	{The result is essentially already well-known, see for instance~\cite[Theorem~5.2,~p.76]{Aubin1993} or \cite[Theorem 6.5.7]{aubin2009set}, so we shall only point out a few details.
		It is well-known that, for each $(x,p)\in \overline{\Omega}\times \mathbb{R}^d$, the infimum in~\eqref{H-moreau-yosida-reg-defn} is attained at a unique point $J_\lambda(x,p)\in\mathbb{R}^d$, and that $H_{\lambda}$ is convex and continuously differentiable with respect to $p$, with partial derivative 
		\begin{equation}
			\frac{\partial H_{\lambda}}{\partial p}(x,p)=\frac{p-J_\lambda(x,p)}{\lambda}\in \partial_p H\left(x,J_\lambda(x,p)\right) \quad\forall (x,p)\in\overline{\Omega}\times\R^\dim,
		\end{equation}
		see e.g.~\cite[Eq.~(55),\,p.~76]{Aubin1993}.
		The continuity of the partial derivative $\frac{\partial H_\lambda}{\partial p}$ over $\overline{\Omega}\times \R^d$ can also be shown using the above hypotheses on $H$ and using the uniqueness of $J_\lambda(x,p)$ for each $(x,p)$.
		Since the subdifferential set $\partial_p H(x,q)$ is contained in the closed ball of radius $L_H$ for any $q\in\R^\dim$, this implies that $\sup_{(x,p)\in \overline{\Omega}\times\R^\dim}\left\lvert\frac{\partial H_{\lambda}}{\partial p}(x,p)\right\rvert\leq L_H$ and thus that~\eqref{H-bounds:lipschitz-lam} holds with the same Lipschitz constant $L_H$.
		To prove~\eqref{moreau-yosida-bounds:reg-unif-approx_1}, first note that $H_\lambda(x,p)\leq H(x,p)$ follows immediately from the definition.
		Also~\eqref{bounds:lipschitz} implies that
		\begin{multline}
			0\leq H(x,p) - H_\lambda(x,p)
			=H(x,p) - H(x,J(x,p))-\frac{\lvert p-J(x,p) \rvert^2}{2\lambda}
			\\ \leq L_H\lvert p-J(x,p)\rvert - \frac{\lvert p-J(x,p) \rvert^2}{2\lambda}
			\leq \frac{L_H^2\lambda }{2},
		\end{multline}
		where the final inequality above follows from Young's inequality.
		This proves~\eqref{moreau-yosida-bounds:reg-unif-approx_1}.}
\end{proof}

{In general, if $H_\lambda$ is defined by Moreau--Yosida regularization~\eqref{H-moreau-yosida-reg-defn}, then $H_\lambda$ is $C^1$-regular with respect to $p$, but $H_\lambda$ is not necessarily $C^2$-regular. To see this, it is enough to consider the example $H(x,p)=\lvert p \rvert$, $p\in \R^\dim$.
In some cases, it is of interest to consider other choices of regularization that can yield higher regularity such as mollification, which we consider below.

\begin{example}
	As a further example, we consider the mollification of the Hamiltonian w.r.t.\ $p$. 
	Let $\rho \in C^r(\R^\dim)$, where $r\in \N\cup\{\infty\}$, denote a nonnegative function with compact support in the unit ball $B_1(0)$, that satisfies $\int_{\mathbb{R}^d}\rho(q)\mathrm{d}q=1$.
	Given $\lambda>0$, let $\rho_{\lambda}\in C_0^{r}(\mathbb{R}^d)$ be the function given by $\rho_{\lambda}(q) \coloneqq \lambda^{-d}\rho(q/\lambda)$, $q\in\mathbb{R}^d$, so that $\int_{\mathbb{R}^d}\rho_{\lambda}(q)\mathrm{d}q=1$. 
	%
	Given $\lambda>0$, let $H_{\lambda}:\overline{\Omega}\times\mathbb{R}^d\to\mathbb{R}$ denote the mollification of $H$ w.r.t.\ $p$ defined by
	\begin{equation}\label{H-mollifcation}
		H_{\lambda}(x,p)\coloneqq \int_{\mathbb{R}^d}H(x,q)\rho_{\lambda}(p-q)\mathrm{d}q.
	\end{equation}
	Lemma~\ref{lemma-mollification-reg-approx-properties} shows that this regularization satisfies \ref{H:reg-family}.
\end{example} 

\begin{lemma}[Mollification-based regularization]\label{lemma-mollification-reg-approx-properties}
	For each $\lambda\in (0,1]$, let $H_{\lambda}$ be defined by \eqref{H-mollifcation}.
	Then, the family $\{H_{\lambda}\}_{\lambda\in (0,1]}$ satisfies \ref{H:reg-family}. 
	In particular, the bound~\eqref{H-bounds:lipschitz-lam} holds with same Lipschitz constant as in~\eqref{bounds:lipschitz} and~\eqref{H-bounds:reg-unif-approx_1} holds with $\omega(\lambda)=C_{\rho}{L_H\lambda}$, i.e.\
	\begin{equation}\label{H-mollification-approx-1}
		\sup_{(x,p)\in \overline{\Omega}\times\mathbb{R}^d}|H_{\lambda}(x,p) - H(x,p)|\leq C_{\rho}L_H\lambda,
	\end{equation}
	where $C_{\rho}\coloneqq \int_{B_1(0)}|q|\rho(q)\mathrm{d}q$, and $H_{\lambda}$ is $C^r$-regular w.r.t.\ $p$.
\end{lemma} 
The proof of this result is based on elementary properties of convolution and so it omitted.
Lemma~\ref{lemma-mollification-reg-approx-properties} shows that mollification can have the advantage of producing smoother regularized Hamiltonians than Moreau--Yosida regularization in some cases.

In addition to the examples above, one can consider various alternatives or improvements, which might be useful for some practical applications.
For instance, in cases where the regularity of the function $p\mapsto H(x,p)$ might depend on $x\in \Omega$, then it might be useful in applications to adapt the regularization to the point $x\in \Omega$, e.g.\ one can consider $H_\lambda(x,p)\coloneqq \int_{\R^d} H(x,q)\rho_{\sigma(x)}(p-q)\mathrm{d}q$ where $\sigma(x)\in (0,\lambda]$ is some chosen function that controls the local regularization for $x\in\Omega$.
}

\subsubsection{Weak formulation of regularized problems}
The following Definition states the notion of weak solution for the regularized problems~\eqref{eq:regularized_MFG}
\begin{definition}[Weak Solution of regularized MFG system]\label{weakdef-moreau-yosida}
	Assume~\ref{H:reg-family}. 
	For each $\lambda\in (0,1]$, a pair $({u}_{\lambda},{m}_{\lambda})\in H_0^1(\Omega)\times H_0^1(\Omega)$ is a weak solution of the regularized MFG system~\eqref{eq:regularized_MFG} if 
	\begin{subequations}\label{weakform-moreau-yosida}
		\begin{align}
			\int_{\Omega}\nu\nabla {u}_{\lambda}\cdot\nabla  \psi+H_{\lambda}[\nabla {u}_{\lambda}]\psi\mathrm{d}x
			=&\langle {F}[{m}_{\lambda}],\psi\rangle_{H^{-1}\times H_0^1} &&\forall \psi\in H^1_0(\Omega),
			\label{weakform1-space-time-moreau-yosida}
			\\ 
			\int_{\Omega}\nu\nabla {m}_{\lambda}\cdot\nabla \phi+{m}_{\lambda}\frac{\partial {H}_{\lambda}}{\partial p}[\nabla {u}_{\lambda}]\cdot\nabla \phi \mathrm{d}x&=\langle {G},\phi\rangle_{H^{-1}\times H_0^1} &&\forall \phi \in H^1_0(\Omega).\label{weakform2-space-time-moreau-yosida} 
		\end{align} 
	\end{subequations}
\end{definition}

{Note that when considering the regularized problems~\eqref{eq:regularized_MFG}, the notions of weak solutions from Definitions~\ref{weakdef} and~\ref{weakdef-moreau-yosida} coincide. This is because there holds $D_pH_{\lambda}[v]=\left\{\frac{\partial H_{\lambda}}{\partial p}(\cdot,\nabla v)\right\}$ for all $v\in W^{1,1}(\Omega)$, since $H_{\lambda}$ is differentiable with respect to the gradient variable.
	
	\begin{remark}[Existence and uniqueness of weak solutions of regularized system \eqref{eq:regularized_MFG}]
		Assuming~\ref{H:reg-family}, we can show that the statement of Theorem~\ref{theorem-existence-uniqueness-mfg-pdi} can be transposed over to the regularized problems~\eqref{eq:regularized_MFG}, under the same assumptions on the data $F$ and $G$. 
		This is simply because Theorem~\ref{theorem-existence-uniqueness-mfg-pdi} can be applied also to the regularized problems, and the notions of weak solution from Definitions~\ref{weakdef} and~\ref{weakdef-moreau-yosida} coincide, as explained above.
		Thus, there exists at least one weak solution of~\eqref{eq:regularized_MFG} in the above sense, and uniqueness also holds under hypotheses~\ref{H:G-postive} and~\ref{H:F-strict-mono}. 
		{However, Example~\ref{ex:nonuniqueness} illustrates how uniqueness of solutions of the regularized problems can also fail in general without the additional hypotheses~\ref{H:G-postive} and~\ref{H:F-strict-mono}.}
\end{remark}}

\section{Main results}\label{sec-main-results}
\subsection{Basic convergence}

The first main result shows that, up to subsequences, the solutions of the regularized problems~\eqref{eq:regularized_MFG} converge to solutions of the MFG PDI~\eqref{mfg-pdi-sys}.
\begin{theorem}[Convergence of solutions of the regularized problems]\label{theorem-convergence-moreau-yosida}
	
	Assume~\ref{H:reg-family}. 
	Suppose that $\{\lambda_j\}_{j\in\mathbb{N}}\subset (0,1]$ is a sequence of real numbers converging to zero, and, {for each $j\in \mathbb{N}$, let $(u_{\lambda_j},m_{\lambda_j})$ denote a weak solution of the regularized problem~\eqref{eq:regularized_MFG} with $\lambda = \lambda_j\in (0,1]$.}
	Then, the sequences $\{u_{\lambda_j}\}_{j\in\mathbb{N}}$, $\{m_{\lambda_j}\}_{j\in\mathbb{N}}$ are uniformly bounded in $H_0^1(\Omega)$. Moreover, there exists a subsequence of $\{(u_{\lambda_j},m_{\lambda_j})\}_{j\in\mathbb{N}}$ (to which we pass without change of notation) and a weak solution $(u,m)$ of~\eqref{mfg-pdi-sys} such that
	\begin{equation}\label{m-conv-2-moreau-yosida}
		\begin{aligned}
			{u}_{\lambda_j} \to u\quad\text{in } H_0^1(\Omega), &&& 
			{m}_{\lambda_j} \rightharpoonup m\quad\text{in }H_0^1(\Omega),
			\\
			{m}_{\lambda_j}\to m \quad \text{in }  \mathcal{X}, &&& {m}_{\lambda_j}\to m \quad \text{in }L^q(\Omega),
		\end{aligned}
	\end{equation}
	as $j\to\infty$, for any $q\in [1,2^*)$, where $2^*\coloneqq \infty$ if $d=2$ and $2^*\coloneqq \frac{2d}{d-2}$ if $d\geq 3$, and for any $q\in[1,\infty]$ if $d=1$. 
\end{theorem}

{Theorem~\ref{theorem-convergence-moreau-yosida} implies that limit points, in the sense of~\eqref{m-conv-2-moreau-yosida}, of weak solutions of the regularized MFG PDE are weak solutions of the MFG PDI~\eqref{mfg-pdi-sys}. 
	This makes precise the sense in which MFG PDI generalize the well-known MFG PDE when relaxing the differentiability condition on the Hamiltonian.
	Another immediate implication of Theorem~\ref{theorem-convergence-moreau-yosida} is that, given an arbitrary neighbourhood (e.g. in the strong topology on $H^1_0(\Omega)\times\mathcal{X}$) of the set of all weak solutions of the MFG PDI~\eqref{mfg-pdi-sys}, then for all $\lambda$ sufficiently small, all weak solutions of the regularized problem~\eqref{eq:regularized_MFG} are contained in the given neighbourhood. This is easily shown by supposing the claim to be false, i.e.\ there would exist a sequence of solutions ${(u_{\lambda_j},m_{\lambda_j})_{j\in \N}}$ of the regularized problems, with $\lambda_j\to 0$ as $j\to \infty$, that are not contained in the neighbourhood, which contradicts the existence of a subsequence that converges to a solution of~\eqref{mfg-pdi-sys} as shown by Theorem~\ref{theorem-convergence-moreau-yosida}.}
%
\begin{remark}[Nonuniqueness of solutions and convergence of subsequences]
{We emphasize that Theorem~\ref{theorem-convergence-moreau-yosida} does not require the assumptions~\ref{H:G-postive} and \ref{H:F-strict-mono}.
	Therefore, in general, the solutions of the original and regularized problems are not necessarily unique, c.f.\ Example~\ref{ex:nonuniqueness} above.
	This is why convergence of the solutions of regularized problems is only shown up to subsequences.}
When solutions of the MFG PDI are nonunique, it is possible that different subsequences of solutions of the regularized problems may converge to different solutions of the MFG PDI, see the example of~Section~\ref{sec-eg-3} below.
\end{remark}

\begin{remark}[Convergence of entire sequence for uniquely solvable MFG PDI]
	{If~\eqref{mfg-pdi-sys} has a unique weak solution, then the convergence in~Theorem~\ref{theorem-convergence-moreau-yosida} holds for the entire sequence.
		As shown in Theorem~\ref{theorem-existence-uniqueness-mfg-pdi}, this includes the case where $G$ satisfies~\ref{H:G-postive} and $F$ satisfies~\ref{H:F-strict-mono}.}
\end{remark}

\begin{remark}[Nonconvergence in the $H^1$-norm for density function approximations]
	The convergence in~\eqref{m-conv-2-moreau-yosida} only states weak convergence in $H^1_0$ of $m_{\lambda_j}$ to $m$.
	This is sharp in general, since it is not always possible to have strong convergence of the density function approximations in the $H^1$-norm. We give an explicit example in~Section~\ref{sec-eg-4} to show this.
	{However, if one assumes some precompactness of the sequence $\left\{\frac{\partial H_{\lambda_j}}{\partial p}[\nabla u_{\lambda_j}]\right\}_{j\in\mathbb{N}}$, then we can recover strong convergence of the gradients of the densities, see Corollary~\ref{corollary-strong-H1-conv-density-1} below.}
\end{remark}

\begin{corollary}[Strong $H_0^1$-pre-compactness for density approximations]\label{corollary-strong-H1-conv-density-1}
	In addition to the hypotheses of Theorem~\ref{theorem-convergence-moreau-yosida}, suppose that  {$\left\{\frac{\partial H_{\lambda_j}}{\partial p}[\nabla u_{\lambda_j}]\right\}_{j\in\mathbb{N}}$ is precompact} in $L^1(\Omega;\mathbb{R}^d)$.
	Then, $\{m_{\lambda_j}\}_{j\in\mathbb{N}}$ is precompact in $H_0^1(\Omega)$.
\end{corollary}

The proofs of Theorem~\ref{theorem-convergence-moreau-yosida} and Corollary~\ref{corollary-strong-H1-conv-density-1} are given in Section~\ref{sec-pfs-of-main-result-1}.

\subsection{Rates of convergence}
In the event that $F$ is strongly monotone, i.e.\ satisfies~\ref{H:F-strong-mono-2}, and $G$ satisfies~\ref{H:G-postive}, we obtain a rate of convergence for the weak solutions of the regularized problem~\eqref{eq:regularized_MFG} to the {unique} weak solution of the MFG PDI~\eqref{mfg-pdi-sys}. The rate of convergence that we derive in this setting is independent of the regularity of the weak solution {of} the MFG PDI~\eqref{weakform}.  
\begin{theorem}[Rate of Convergence of solutions to regularized problems]\label{theorem-rate-of-convergence-moreau-yosida}
	{Assume the hypotheses~\ref{H:G-postive},~\ref{H:F-strong-mono-2},~\ref{H:F-Lipschitz-continuous}, and~\ref{H:reg-family}.
		Let $(u,m)$ and $(u_{\lambda},m_{\lambda})$, $\lambda\in(0,1]$, be the respective unique solutions of~\eqref{weakform} and~\eqref{weakform-moreau-yosida}.
		Then
		\begin{equation}\label{rates-of-convergence-eqns-density}
			\norm{u-u_{\lambda}}_{H^1(\Omega)}+\norm{m-m_{\lambda}}_{\mathcal{X}}\lesssim \omega(\lambda)^{\frac{1}{2}},
		\end{equation}
		for all $\lambda$ sufficiently small, where the hidden constant depends only on $\Omega$, $\nu$, $\dim$, $L_H$, $L_F$, $\sup_{\sigma\in[0,1]}\omega(\sigma)$, and $c_F$.}
\end{theorem}

{The proof of Theorem~\ref{theorem-rate-of-convergence-moreau-yosida} is given in Section~\ref{sec-pfs-of-main-result-2}.
	We emphasize that there is no assumption of higher regularity of the solution $(u,m)$ in Theorem~\ref{theorem-rate-of-convergence-moreau-yosida}, which is important given that it is known from examples that~$(u,m)$ may have limited smoothness, c.f.~\cite[Section~3.3]{osborne2022analysis}.
	For instance, in the case of Moreau--Yosida regularization, c.f.~\eqref{H-moreau-yosida-reg-defn},  {or mollification-based regularization, c.f.~\eqref{H-mollification-approx-1},} Theorem~\ref{theorem-rate-of-convergence-moreau-yosida} implies a rate of convergence of order $\frac{1}{2}$ with respect to $\lambda$.
	The bound~\eqref{rates-of-convergence-eqns-density} {will} play an important role in {future work on} the design and analysis of numerical methods for solving the MFG PDI system, since it allows one to approximate the solution by that of a regularized problem, with some quantitative control on the error.}

\section{Preliminary results and proof of Theorem~\ref{theorem-existence-uniqueness-mfg-pdi}}\label{sec-analysis-prelim-results}

{We start by gathering some preparatory lemmas. The following Lemma is from~\cite[Lemma~4.3]{osborne2022analysis}, and shows that the set-valued map $D_pH$ has nonempty images in $L^\infty(\Omega;\R^\dim)$ that are uniformly bounded in the closed ball of radius $L_H$.}

\begin{lemma}[\cite{osborne2022analysis}]\label{lemma-DpH-non-empty}
	{The} set-valued map $D_pH:W^{1,1}(\Omega)\rightrightarrows L^{\infty}(\Omega;\mathbb{R}^d)$ possesses nonempty images and we have the bound
	\begin{equation}\label{Hsubdiff-bound}
		\sup_{v\in W^{1,1}(\Omega)}\sup_{\tilde{b}\in D_pH[v]}\|\tilde{b}\|_{L^{\infty}(\Omega;\mathbb{R}^d)}\leq L_H.
	\end{equation}
\end{lemma}

{The next Lemma is from~\cite[Lemma~4.4]{osborne2022analysis}, where it is shown that $D_pH$ has closed graph when $L^\infty(\Omega;\R^d)$ is equipped with its weak-$*$ topology.}

\begin{lemma}[\cite{osborne2022analysis}]\label{inclusion}
	{Let $H$ be the function given by~\eqref{Hamiltonian}. Suppose $\{v_j\}_{j\in\mathbb{N}}\subset H^1(\Omega)$, $\{\tilde{b}_{j}\}_{j\in\mathbb{N}}\subset L^{\infty}(\Omega;\mathbb{R}^d) $ are sequences such that $\tilde{b}_j\in D_pH[v_j]$ for all $j\in\mathbb{N}$. If $v_j\to v$ in $H^1(\Omega)$ and $\tilde{b}_j\rightharpoonup^* \tilde{b}$ in $L^{\infty}(\Omega;\mathbb{R}^d)$ as $j\to \infty$, then $\tilde{b}\in D_pH[v]$.} 
\end{lemma}

{The following Lemma relates the partial derivatives of $H_\lambda$ with the subdifferential of $H$ when considering compositions with gradients of functions that form a convergent sequence in $H^1$.}
\begin{lemma}\label{lemma-reg-inclusion}
	Assume~\ref{H:reg-family} holds. Let $\{\lambda_j\}_{j\in\mathbb{N}}\subset (0,1]$ be a {sequence that converges to $0$ as $j\to \infty$,} and let $\{{v}_j\}_{j\in\mathbb{N}}\subset H^1(\Omega)$ be a sequence {of functions in $H^1(\Omega)$ that converges to a limit $v\in H^1(\Omega)$.} 
	{Then, there exists a subsequence of $\{\frac{\partial H_{\lambda_{j}}}{\partial p}[\nabla {v}_{j}]\}_{j\in\mathbb{N}}$ that is weakly-$*$ convergent in $L^\infty(\Omega;\R^d)$ to a limit $\tilde{b}_*\in D_pH[v]$.}
\end{lemma}
\begin{proof}
	{For each $q\in \R^\dim$, and each $j\in \N$, we define the} function $\omega_{j,q}:\Omega\to\mathbb{R}$ given by
	\begin{equation}
		\begin{aligned}
			\omega_{j,q}(x)\coloneqq H_{\lambda_j}(x,q)- H_{\lambda_j}(x,\nabla v_j(x))-\frac{\partial H_{\lambda_j}}{\partial p}(x,\nabla v_j(x))\cdot(q-\nabla v_j(x)) , \quad x \in \Omega,
		\end{aligned}
	\end{equation}
	By hypothesis~\ref{H:reg-family}, for each $j\in\mathbb{N}$ and $x\in\overline{\Omega}$ the function $\mathbb{R}^d\ni p\mapsto H_{\lambda_j}(x,p)$ is convex, and the partial derivative $\frac{\partial H_{\lambda_j}}{\partial p}:\Omega\times\mathbb{R}^d\to\mathbb{R}^d$ exists and is continuous.
	{Therefore, the function} $\omega_{j,q}\geq 0$ a.e.\ in $\Omega$ for all $j\in\mathbb{N}$.
	Moreover, since $H_{\lambda_j}$ satisfies the Lipschitz condition~\eqref{H-bounds:lipschitz-lam}, we have that $\frac{\partial H_{\lambda_j}}{\partial p}(\cdot,\nabla v_j)\in L^{\infty}(\Omega;\mathbb{R}^d)$ with $\lVert \frac{\partial H_{\lambda_j}}{\partial p}(\cdot,\nabla v_j) \rVert_{L^{\infty}(\Omega;\mathbb{R}^d)}\leq L_H$.
	It then follows from the definition of $\omega_{j,q}$ that $\omega_{j,q}\in L^2(\Omega)$ for all $j\in\mathbb{N}$. Since $L^1(\Omega;\mathbb{R}^d)$ is separable, the closed ball in $L^{\infty}(\Omega;\mathbb{R}^d)$ is weak-$*$ sequentially compact. Therefore, we may pass to a subsequence (without change of notation) such that $\frac{\partial H_{\lambda_j}}{\partial p}(\cdot,\nabla v_j)\rightharpoonup^*\tilde{b}_*$ in $L^{\infty}(\Omega;\mathbb{R}^d)$ as $j\to\infty$ for some $\tilde{b}_*\in L^{\infty}(\Omega;\mathbb{R}^d)$.
	The bound~\eqref{H-bounds:reg-unif-approx_1} in~\ref{H:reg-family}, together with the strong convergence of $\{v_j\}_{j\in\mathbb{N}}$ to $v$ in $H^1(\Omega)$, allow us to deduce that $\omega_{j,q}$ converges weakly in $L^2(\Omega)$ to the function $\omega_{q}\in L^2(\Omega)$ that is defined by
	\begin{equation}
		\omega_{q}(x)\coloneqq H(x,q)- H(x,\nabla v(x))-\tilde{b}_*(x)\cdot(q-\nabla v(x)), \quad x\in \Omega.
	\end{equation}
	{Mazur's theorem then implies that $\omega_q\geq 0$ a.e.\ in $\Omega$, for each $q\in \R^\dim$, since $\omega_{q,j}\rightharpoonup \omega_q$ in $L^2(\Omega)$ as $j\tends \infty$, with each $\omega_{q,j}$ nonnegative a.e.\ in $\Omega$.}
	Since $q\in\mathbb{R}^d$ was arbitrary and since $\mathbb{R}^d$ is separable, after possibly excising a set of measure zero, we conclude that, for a.e.\ $x\in\Omega$, 
	\begin{equation} 
		H(x,q)- H(x,\nabla v(x))-\tilde{b}_*(x)\cdot(q-\nabla v(x))\geq 0\quad\forall q\in\mathbb{R}^d.
	\end{equation}
	{This implies that $\tilde{b}_*\in D_pH[v]$ and concludes the proof.}
\end{proof}

Let $\mathcal{G}(L_H)$ denote the set of all operators $L:H_0^1(\Omega)\to H^{-1}(\Omega)$ of the form
\begin{equation}\label{eq:wmp_1}
	\langle Lu,v\rangle_{H^{-1}\times H_0^1}=\int_{\Omega}\nu\nabla u\cdot\nabla v+\tilde{b}\cdot\nabla u v\,\mathrm{d}x, 
\end{equation} 
where $\tilde{b}:\Omega\to\mathbb{R}^d$ is some vector field satisfying $\lVert \tilde{b}\rVert_{L^{\infty}(\Omega;\mathbb{R}^d)}\leq L_H.$
In addition, given an operator $L\in \mathcal{G}(L_H)$, we define $L^*:H_0^1(\Omega)\to H^{-1}(\Omega)$, the formal adjoint of $L$, by $\langle L^*w,v\rangle_{H^{-1}\times H_0^1} \coloneqq \langle Lv, w\rangle_{H^{-1}\times H_0^1}$ for all $ w,v\in H_0^1(\Omega).$ 
In the analysis we will use the following uniform invertibility result which is an application of \cite[Lemma 4.5]{osborne2022analysis}. 
\begin{lemma}\label{lemma-uniform-L-inv-bound}
	For every operator $L\in\mathcal{G}(L_H)$, both $L$ and $L^*$ are boundedly invertible as mappings from $H^1_0(\Omega)$ to $H^{-1}(\Omega)$, and there exists a constant $C_1> 0$ depending on only $\Omega$, $d$, $\nu$, and  $L_H$ such that 
	\begin{equation}\label{eq:invertibility-continuous} 
		\sup_{L\in\mathcal{G}(L_H)}\max\left\{\left\|L^{-1}\right\|_{\mathcal{L}\left(H^{-1}(\Omega),H_0^1(\Omega)\right)},\left\|{L^*}^{-1}\right\|_{\mathcal{L}\left(H^{-1}(\Omega),H_0^1(\Omega)\right)}\right\}\leq C_1.
	\end{equation}
\end{lemma}
We note the fact that each operator from $L\in \mathcal{G}(L_H)$ and its adjoint $L^*$ are both invertible follows from the Fredholm Alternative and the Weak Maximum Principle (WMP). Importantly, each $L\in \mathcal{G}(L_H)$ satisfies the conditions of the Weak Maximum Principle~\cite[Theorem~8.1]{gilbarg2015elliptic}, which implies also the Comparison Principle for the adjoint operator $L^*$, so that $L^*v\geq 0$ in the sense of distributions in $H^{-1}(\Omega)$ implies $v\geq 0$ a.e.\ in $\Omega$.

{We now state the following well-posedness result for the HJB equation and its regularizations. To express this as succinctly as possible, let us define $H_0\coloneqq H$ where $H$ is the Hamiltonian defined in~\eqref{Hamiltonian}, so that we may then extend the family of Hamiltonians considered to the set $\{H_\lambda\}_{\lambda\in[0,1]}$.}

\begin{lemma}[Well-posedness of the Regularized HJB equation]\label{lemma-HJB-continuous-moreau-yosida}
	{Assume~\ref{H:reg-family}.
		Then, for each $\lambda\in [0,1]$ and each $\widetilde{m}\in \mathcal{X}$, there exists a unique $u_{\lambda}\in H_0^1(\Omega)$ such that
		\begin{equation}\label{eq:weak-hjb-reg}
			\int_{\Omega}\nu\nabla u_{\lambda}\cdot\nabla \psi+ H_{\lambda}[\nabla u_{\lambda}]\psi\,\mathrm{d}x = \langle F[\widetilde{m}],\psi\rangle_{H^{-1}\times H_0^1} \quad\forall \psi\in H_0^{1}(\Omega).
		\end{equation}
		{In addition,} 
		\begin{equation}\label{eq:hjb-reg_a_priori_bound}
			\lVert u_{\lambda} \rVert_{H^1(\Omega)}\lesssim \lVert \widetilde{m} \rVert_{\mathcal{X}} + \lVert f \rVert_{C(\overline{\Omega}\times \mathcal{A})}+\omega(\lambda)+1.
		\end{equation}
		Moreover, the solution $u_{\lambda}$ depends continuously on $\widetilde{m}$, i.e., if $\left\{\widetilde{m}_j\right\}_{j\in\mathbb{N}}\subset \mathcal{X}$ is such that $\widetilde{m}_j\to \widetilde{m}$ in $\mathcal{X}$ as $j\to\infty$, then the corresponding sequence of solutions $\left\{u_{\lambda,j}\right\}_{j\in\mathbb{N}}\subset H_0^1(\Omega)$ of~\eqref{eq:weak-hjb-reg} with data $\widetilde{m}_j$ converges in $H_0^1(\Omega)$ to the unique solution $u_{\lambda}$ of~\eqref{eq:weak-hjb-reg} with datum $\widetilde{m}$.}
\end{lemma}

{
	Note that the only assumptions on $F$ required by Lemma~\ref{lemma-HJB-continuous-moreau-yosida} is that $F$ should be continuous from $\mathcal{X}$ to $H^{-1}(\Omega)$ and satisfy the growth condition~\eqref{F-linear-growth}.
	Observe also that for the case $\lambda=0$, Lemma~\ref{lemma-HJB-continuous-moreau-yosida} provides a statement of well-posedness for the unregularized HJB equation
	\begin{equation}\label{eq:weak-hjb}
		\int_{\Omega}\nu\nabla u\cdot\nabla \psi+ H[\nabla u]\psi\,\mathrm{d}x = \langle F[\widetilde{m}],\psi\rangle_{H^{-1}\times H_0^1} \quad\forall \psi\in H_0^{1}(\Omega),
	\end{equation}
	and moreover the term $\omega(\lambda)$ in the bound~\eqref{eq:hjb-reg_a_priori_bound} then vanishes as a result of the hypothesis $\omega(0)=0$.
	Lemma~\ref{lemma-HJB-continuous-moreau-yosida} is essentially already well-known, and is a small generalization of~\cite[Lemma~4.6]{osborne2022analysis}. Since its proof is standard by now, we shall omit it here for brevity, see \cite[Appendix~A]{osborne2022analysis} for further details.}

\subsection{Proof of Theorem~\ref{theorem-existence-uniqueness-mfg-pdi}}\label{sec:proof_existence}
The existence of a weak solution of the MFG PDI~\eqref{mfg-pdi-sys} {can be shown using Kakutani's fixed point theorem}. 
It follows the ideas in~\cite[Theorem~3.2.1]{osborne2024thesis} and in~\cite{ducasse2020second}. 
{We start by recalling Kakutani's fixed point theorem for set-valued maps on locally convex spaces, see for instance~\cite[Chapter~9; Theorem~9.B]{ZMR0816732}.}
\begin{theorem}[Kakutani's fixed point theorem]\label{thm-kakutani}
	{Suppose that} 
	\begin{enumerate}
		\item {$\mathcal{B}$ is a nonempty, compact, convex set in a locally convex space $\mathcal{Y}$;}
		\item {$\mathcal{V}:\mathcal{B}\rightrightarrows\mathcal{B}$ is a set-valued map such that $\mathcal{V}[\tilde{b}]$ is nonempty, closed and convex for all $\tilde{b}\in\mathcal{B}$; and}
		\item {$\mathcal{V}$ is upper semi-continuous.}
	\end{enumerate}
	{Then $\mathcal{V}$ has a fixed point: there exists a $\tilde{b}\in\mathcal{B}$ such that $\tilde{b}\in \mathcal{V}[\tilde{b}]$.} 
\end{theorem}

\begin{proof}[Proof of Theorem~\ref{theorem-existence-uniqueness-mfg-pdi}]
	{Recall that the Lipschitz constant of the Hamiltonian is $L_H$ given in~\eqref{bounds:lipschitz}. We equip the space $\mathcal{Y}\coloneqq L^{\infty}(\Omega;\mathbb{R}^d)$ with its weak-$*$ topology, noting that it is then a locally convex topological vector space. Let $\mathcal{B}$ denote the ball 
		\begin{equation}
			\mathcal{B}\coloneqq \left\{\tilde{b}\in L^{\infty}(\Omega;\mathbb{R}^d): \|\tilde{b}\|_{L^{\infty}(\Omega;\mathbb{R}^d)}\leq L_H\right\}.
		\end{equation}
		We note that $\mathcal{B}$ is nonempty, closed in the weak-$*$ topology, and convex. Since $L^1(\Omega;\mathbb{R}^d)$ is separable, the weak-$*$ topology on $\mathcal{B}$ is metrizable \cite[Ch.\ 15]{royden2018real}. Moreover, Helly's theorem implies that $\mathcal{B}$ is compact.
		
		Let $M: \mathcal{B}\to H_0^1(\Omega)$ be the map defined as follows: for each $\tilde{b}\in \mathcal{B}$, let $M[\tilde{b}]$ in $H_0^1(\Omega)$ be the unique solution of
		\begin{equation}\label{M-map-def}
			\int_{\Omega}\nu\nabla M[\tilde{b}]\cdot\nabla \phi+M[\tilde{b}]\tilde{b}\cdot\nabla \phi\,\mathrm{d}x=\langle G,\phi\rangle_{H^{-1}\times H_0^1}\text{  }\text{ }\forall\phi\in H_0^1(\Omega).
		\end{equation}
		The map $M$ is well-defined thanks to Lemma~\ref{lemma-uniform-L-inv-bound}. 
		Next, let $U:{\mathcal{X}}\to H_0^1(\Omega)$ be the map defined as follows: for each {$\widetilde{m}\in \mathcal{X}$, let $U[\widetilde{m}]\in H_0^1(\Omega)$} denote the unique solution of
		\begin{equation}\label{U-map-def}
			\int_{\Omega}\nu\nabla U[m]\cdot\nabla \psi+H(x,\nabla U[m])\psi \,\mathrm{d}x=\langle F[{\widetilde{m}}],\psi\rangle_{H^{-1}\times H_0^1} \text{ }\text{ }\forall\psi\in H_0^1(\Omega).
		\end{equation}
		The map $U$ is well-defined by Lemma~\ref{lemma-HJB-continuous-moreau-yosida}. 
		
		Now, we define the set-valued map $\mathcal{V}:\mathcal{B}\rightrightarrows L^{\infty}(\Omega;\mathbb{R}^d)$ as follows: for each $\tilde{b}\in \mathcal{B}$, let 
		\begin{equation}
			\mathcal{V}[\tilde{b}]\coloneqq  D_pH\big[U\big[M\big[\tilde{b}\big]\big]\big].
		\end{equation}
		{Note that in taking the composition of $U$ with $M$, we are implicitly using the continuous embedding of $H^1_0(\Omega)$ into $\mathcal{X}$.}
		{The existence of a weak solution {of} the MFG PDI~\eqref{mfg-pdi-sys} in the sense of Definition~\ref{weakdef} is equivalent to showing the existence of a fixed point of $\mathcal{V}$, i.e.\ that there exists a $\tilde{b}_*\in \mathcal{B}$ such that $\tilde{b}_*\in\mathcal{V}[\tilde{b}_*].$ Indeed, if $\tilde{b}_*\in \mathcal{B}$ satisfies $\tilde{b}_*\in\mathcal{V}[\tilde{b}_*]$ then a solution pair $(u,m)$ of the weak formulation~\eqref{weakform} of~\eqref{mfg-pdi-sys} is given by $m\coloneqq M[\tilde{b}_*]$ and $u\coloneqq U[m]$ with $\tilde{b}_*\in D_pH[u]$, while the converse is obvious.}
		
		{We now check that $\mathcal{V}$ satisfies all the conditions of Kakutani's fixed point theorem. 
			First,} Lemma~\ref{lemma-DpH-non-empty} implies that $\mathcal{V}[\tilde{b}]\subset\mathcal{B}$ for each $\tilde{b}\in\mathcal{B}$, so $\mathcal{V}:\mathcal{B}\rightrightarrows \mathcal{B}$. Moreover, for every $\tilde{b}\in \mathcal{B}$, the set $\mathcal{V}[\tilde{b}]$ is nonempty and convex.
		Indeed, for each $\tilde{b}\in\mathcal{B}$ the set $\mathcal{V}[\tilde{b}]$ is nonempty by Lemma~\ref{lemma-DpH-non-empty}.
		{The fact that $\mathcal{V}$ has convex images is an elementary consequence of the fact that $\partial_p H$ has convex images.}
		Indeed, let $b_1,b_2\in \mathcal{V}[\tilde{b}]$ and $\theta\in [0,1]$ be given. It follows from the definition of the inclusions $b_j\in \mathcal{V}[\tilde{b}] = D_pH\big[U\big[M\big[\tilde{b}\big]\big]\big]$ for $j\in\{1,2\}$ and the subdifferential $\partial_pH$ that, for a.e.\ $x\in \Omega$, 
		$$H(x,q)\geq H(x,\nabla U\big[M\big[\tilde{b}\big]\big](x)) + b_j(x)\cdot(q-\nabla U\big[M\big[\tilde{b}\big]\big](x))\quad\forall q\in\mathbb{R}^d, j\in\{1,2\}$$
		from which we see that
		$$H(x,q)\geq H(x,\nabla U\big[M\big[\tilde{b}\big]\big](x)) + (\theta b_1(x)+(1-\theta)b_2(x))\cdot(q-\nabla U\big[M\big[\tilde{b}\big]\big](x))\quad\forall q\in\mathbb{R}^d.$$ This shows that $\theta b_1+(1-\theta)b_2\in \mathcal{V}[\tilde{b}]$, so $\mathcal{V}[\tilde{b}]$ is convex as claimed.
		Furthermore, $\mathcal{V}[\tilde{b}]$ is closed for all $\tilde{b}\in\mathcal{B}$. Indeed, since $\mathcal{B}$ is metrisable and $\mathcal{V}[\tilde{b}]\subset\mathcal{B}$, to show that $\mathcal{V}[\tilde{b}]$ is closed it suffices to show that any sequence $\{\overline{b}_j\}_{j\in\mathbb{N}}\subset \mathcal{V}[\tilde{b}]$ with $\overline{b}_j\rightharpoonup^*\overline{b}$ in $L^{\infty}(\Omega;\mathbb{R}^d)$ as $j\to\infty$ satisfies $\overline{b}\in \mathcal{V}[\tilde{b}]=D_pH\big[U\big[M\big[\tilde{b}\big]\big]\big]$. But, given such a sequence $\{\overline{b}_j\}_{j\in\mathbb{N}}$, one can consider the constant sequence  $\{v_j\}_{j\in\mathbb{N}}$ in $H^1(\Omega)$ defined by $v_j\coloneqq U\big[M\big[\tilde{b}\big]\big]$ in $H^1(\Omega)$ for $j\in\mathbb{N}$, and then apply Lemma~\ref{inclusion} to deduce the required inclusion $\overline{b}\in D_pH\big[U\big[M\big[\tilde{b}\big]\big]\big]=\mathcal{V}[\tilde{b}],$ so $\mathcal{V}[\tilde{b}]$ is closed.
		
		{It remains only to} verify that $\mathcal{V}$ is upper semi-continuous. To this end, it suffices to prove that the graph of $\mathcal{V}$ is closed; c.f.\ \cite[Ch.\ 1, Corollary 1, p.\ 42]{MR0755330}.  
		Let $\mathcal{W}$ denote the graph of $\mathcal{V}$, which is defined by 
		\begin{equation}\label{graph-def}
			\mathcal{W}\coloneqq \left\{(\tilde{b},\overline{b})\in \mathcal{B}\times\mathcal{B}:\overline{b}\in\mathcal{V}[\tilde{b}]\right\}.
		\end{equation}
		Since $\mathcal{B}$ is metrisable, to show that the graph $\mathcal{W}$ is a closed it is enough to show that whenever a sequence $\{(\tilde{b}_i,\overline{b}_i)\}_{i\in\mathbb{N}}\subset \mathcal{W}$ converges weakly-$*$ in $\mathcal{B}\times\mathcal{B}$ to a point $(\tilde{b},\overline{b})$ as $i\to\infty$, then $(\tilde{b},\overline{b})\in\mathcal{W}$,  which is equivalent to $\overline{b}\in \mathcal{V}[\tilde{b}]$. 
		Let us then suppose that we are given a sequence $\{(\tilde{b}_i,\overline{b}_i)\}_{i\in\mathbb{N}}\subset \mathcal{W}$ that converges weakly-$*$ in $\mathcal{B}\times\mathcal{B}$ to a point $(\tilde{b},\overline{b})$ as $i\to\infty$. 
		To begin, we claim that $M[\tilde{b}_i]\to M[\tilde{b}]$ in $\mathcal{X}$ as $i\to\infty$.
		Indeed, since $\{\tilde{b}_i\}_{i\in\mathbb{N}}\subset\mathcal{B}$, for each $i\in\mathbb{N}$ we apply Lemma~\ref{lemma-uniform-L-inv-bound} to obtain the uniform bound
		\begin{equation}\label{M_n}
			\sup_{i\in\mathbb{N}}\|M[\tilde{b}_i]\|_{H^1(\Omega)}\lesssim \|G\|_{H^{-1}(\Omega)}.
		\end{equation}
		
		We deduce from this that any given subsequence $\{M[\tilde{b}_{i_j}]\}_{j\in\mathbb{N}}$ is bounded uniformly in $H_0^1(\Omega)$.
		The Rellich--Kondrachov compactness theorem {and the compactness of the embedding of $H^1_0(\Omega)$ into $\mathcal{X}$} then {imply} that there exists a further subsequence $\{M[\tilde{b}_{i_{j_k}}]\}_{k\in\mathbb{N}}$ and {a} $m\in H_0^1(\Omega)$ {such that the $M[\tilde{b}_{i_{j_k}}]$ converge to $m$ weakly in $H^1_0(\Omega)$, strongly in $L^2(\Omega)$, and also strongly in $\mathcal{X}$, as $k\to \infty$.}
		By $L^{\infty}$-weak-$*$ $\times$ $L^2$-strong convergence, we also have that $M[\tilde{b}_{i_{j_k}}]\tilde{b}_{i_{j_k}}\rightharpoonup m\tilde{b}$ in $L^2(\Omega;\mathbb{R}^d)$ as $k\to\infty$. 
		Passing to the limit in the KFP equation~\eqref{M-map-def} satisfied by $M[\tilde{b}_{i_{j_k}}]$ for $k\in\mathbb{N}$, we deduce that $m$ satisfies 
		\begin{equation}\label{m-alt-eqn}
			\begin{aligned}
				\int_{\Omega}\nu\nabla m\cdot\nabla \phi+m\tilde{b}\cdot\nabla \phi\,\mathrm{d}x=\langle G,\phi\rangle_{H^{-1}\times H_0^1} &&&\forall\phi\in H_0^1(\Omega).
			\end{aligned}
		\end{equation}
		But by definition of $M[\tilde{b}]$ in~\eqref{M-map-def}, we see that $m=M[\tilde{b}]$ in $H_0^1(\Omega)$.
		{The uniqueness of the limit then} implies that the entire sequence $\{M[\tilde{b}_i]\}_{i\in\mathbb{N}}$ satisfies $M[\tilde{b}_i]\to M[\tilde{b}]$ in $\mathcal{X}$ as $i\to\infty$. 
		Lemma~\ref{lemma-HJB-continuous-moreau-yosida} implies that $U[M[\tilde{b}_i]]\to U[M[\tilde{b}]]$ in $H_0^1(\Omega)$ as $i\to\infty$. 
		By hypothesis, $\overline{b}_i\in \mathcal{V}[\tilde{b}_i]=D_pH[U[M[\tilde{b}_i]]]$ for $i\in\mathbb{N}$ and $\overline{b}_i\rightharpoonup^* \overline{b}$ in $L^{\infty}(\Omega;\mathbb{R}^d)$ as $i\to\infty$. 
		We conclude from Lemma~\ref{inclusion} that $\overline{b}\in D_pH[U[M[\tilde{b}]]]$, i.e.\ $\overline{b}\in \mathcal{V}[\tilde{b}]$.
		We have therefore shown that {the graph} $\mathcal{W}$ is closed, so $\mathcal{V}$ is upper semi-continuous.
		
		We have thus shown that the map $\mathcal{V}:\mathcal{B}\rightrightarrows\mathcal{B}$ satisfies the conditions of Kakutani's fixed-point theorem, so $\mathcal{V}$ admits a fixed point and therefore there exists a weak solution {of} the MFG PDI~\eqref{mfg-pdi-sys} in the sense of Definition~\ref{weakdef}.}
	
	Assuming, in addition, that $F$ satisfies~\ref{H:F-strict-mono} and $G$ satisfies~\ref{H:G-postive}, the uniqueness of a weak solution {of} the MFG PDI~\eqref{mfg-pdi-sys} in the sense of Definition~\ref{weakdef} follows from the same argument used in the proof of \cite[Theorem 3.4]{osborne2022analysis}.
	This argument carries through thanks to the convexity of the Hamiltonian $H$ w.r.t.\ $p$, together with the nonnegativity of $G$ as a distribution in $H^{-1}(\Omega)$ and the Comparison Principle ensuring that the density $m$ is nonnegative almost everywhere. 
\end{proof}

\section{Proofs of Theorem~\ref{theorem-convergence-moreau-yosida} and Corollary~\ref{corollary-strong-H1-conv-density-1}}\label{sec-pfs-of-main-result-1}

\subsection{Proof of Theorem~\ref{theorem-convergence-moreau-yosida}}

Let $\{(u_{\lambda_j},m_{\lambda_j})\}_{j\in\mathbb{N}}$ denote a sequence defined as follows: for each $j\in\mathbb{N}$, let $(u_{\lambda_j},m_{\lambda_j})$ denote a weak solution {of} the regularized problem~\eqref{eq:regularized_MFG} with $\lambda = \lambda_j\in (0,1]$. Thanks to Lemma~\ref{lemma-uniform-L-inv-bound}, the resulting sequence $\{m_{\lambda_j}\}_{j\in\mathbb{N}}$ is uniformly bounded in the $H^1$-norm. Since $\lambda_j\to 0$ as $j\to\infty$ and the embedding $H_0^1(\Omega)\subset\mathcal{X}$ is continuous, the bound~\eqref{eq:hjb-reg_a_priori_bound} and the definition of the resulting sequence $\{u_{\lambda_j}\}_{j\in\mathbb{N}}$ imply that $\{u_{\lambda_j}\}_{j\in\mathbb{N}}$ is uniformly bounded in the $H^1$-norm.  We may pass to subsequences, without change of notation, that satisfy as $j\to \infty$
\begin{align}
	m_{\lambda_j}\rightharpoonup m\quad\text{in}\quad H_0^1(\Omega),\quad m_{\lambda_j}\to m\quad&\text{in}\quad L^q({\Omega})\label{m-lam-conv-1},
	\\
	u_{\lambda_j}\rightharpoonup u\quad\text{in}\quad H_0^1(\Omega),\quad u_{\lambda_j}\to u\quad&\text{in}\quad L^q({\Omega})\label{u-lam-conv-1},
\end{align} 
for some $m,u\in H_0^1(\Omega)$,  for any $q\in [1,2^*)$, where $2^*\coloneqq \infty$ if $d=2$ and $2^*\coloneqq \frac{2d}{d-2}$ if $d\geq 3$, and for any $q\in[1,\infty]$ if $d=1$. Since the embedding  $H_0^1(\Omega)\subset\mathcal{X}$ is also compact, we may pass to a further subsequence (without change of notation) such that
\begin{equation}\label{m-lam-conv-2}
	m_{\lambda_j}\to m\quad\text{in}\quad \mathcal{X} 
\end{equation}as $j\to\infty$.

The  boundedness of $\{u_{\lambda_j}\}_{j\in\mathbb{N}}$ in $H_0^1(\Omega)$, together with the linear growth of the Hamiltonian $H_{\lambda_j}=H_{\lambda_j}(\cdot,p)$ (see~\eqref{H-bounds:linear-growth-lam}), imply that the sequence $\{H_{\lambda_j}[\nabla u_{\lambda_j}]\}_{j\in\mathbb{N}}$ is uniformly bounded in $L^2(\Omega)$. Therefore, there exists $g\in L^2(\Omega)$ such that, by passing to a subsequence without change of notation, we have
\begin{equation}\label{weakHconv}
	H_{\lambda_j}[\nabla u_{\lambda_j}]\rightharpoonup g\quad\text{in}\quad L^2(\Omega)
\end{equation}
as $j\to\infty$.
The continuity of $F:\mathcal{X}\to H^{-1}(\Omega)$ and the convergence~\eqref{m-lam-conv-2} imply that $F[m_{\lambda_j}]\to F[m]$ in $H^{-1}(\Omega)$ as $j\to\infty$. 
This convergence, together with~\eqref{u-lam-conv-1} and~\eqref{weakHconv}, {allows us 
	pass to the limit $j\tends \infty$ in the regularized HJB equation~\eqref{weakform1-space-time-moreau-yosida} to find that}
\begin{equation}\label{ueqnn}
	\int_{\Omega}\nu\nabla u\cdot\nabla v+gv\,\mathrm{d}x=\langle F[m],v\rangle_{H^{-1}\times H_0^1}\quad \forall v\in H_0^1(\Omega).
\end{equation}

We conclude  that $u_{\lambda_j}\to u$ in $H_0^1(\Omega)$ as $j\to \infty$. 
Indeed, {for each $j\in \N$, by testing~\eqref{weakform1-space-time-moreau-yosida} with $u_{\lambda_j}$, we see that}
\begin{multline}\label{eq:u-lam-j-norm-limit}
	\lim_{j\to\infty}\|\nabla u_{\lambda_j}\|_{L^2(\Omega)}^2=\lim_{j\to\infty}\nu^{-1}\left[\langle F[m_{\lambda_j}],u_{\lambda_j}\rangle_{H^{-1}\times H_0^1} - \int_{\Omega}H_{\lambda_j}[\nabla u_{\lambda_j}]u_{\lambda_j}\mathrm{d}x\right]
	\\
	=\nu^{-1}\left[\langle F[m],u\rangle_{H^{-1}\times H_0^1} - \int_{\Omega}gu\mathrm{d}x\right]=\|\nabla u\|_{L^2(\Omega)}^2,
\end{multline}
{where the final identity above follows from~\eqref{ueqnn} with test function $v=u$.}
{It then follows from~\eqref{u-lam-conv-1} and~\eqref{eq:u-lam-j-norm-limit} that the convergence of $u_{\lambda_j}\to u$ is also strong in $H_0^1(\Omega)$ as $j\to\infty$.}

{Next, it follows from~\eqref{bounds:lipschitz}, \eqref{H-bounds:reg-unif-approx_1} and from~\eqref{weakHconv} that $g=H[\nabla u]$ in $\Omega$ and thus}~\eqref{ueqnn} implies that $u$ solves 
\begin{equation}\label{u:eqn}
	\int_{\Omega}\nu\nabla u\cdot\nabla \psi+H[\nabla u]\psi\,\mathrm{d}x=\langle F[m],\psi\rangle_{H^{-1}\times H_0^1}\quad \forall \psi\in H_0^1(\Omega).
\end{equation}

Next, we deduce from Lemma~\ref{lemma-reg-inclusion} that we can pass to a subsequence (without change of notation) such that $\frac{\partial H_{\lambda_j}}{\partial p}[\nabla u_{\lambda_j}]\rightharpoonup^*\tilde{b}_*$ in $L^{\infty}(\Omega;\mathbb{R}^d)$, as $j\to\infty$, for some $\tilde{b}_*\in D_pH[u]$.
{We then use~\eqref{m-lam-conv-1} to pass to the limit in~\eqref{weakform2-space-time-moreau-yosida} to obtain}
\begin{equation}\label{weak-KFP-pf}
	\int_{\Omega}\nu\nabla m\cdot\nabla \phi+m\tilde{b}_*\cdot\nabla \phi \text{ }\mathrm{d}x=\langle G,\phi\rangle_{H^{-1}\times H_0^1}\quad\forall \phi\in H_0^1(\Omega).
\end{equation}
{This shows that the pair $(u,m)$ solves~\eqref{weakform} and moreover that, up to a subsequence, we have~\eqref{m-conv-2-moreau-yosida}.}
\hfill\proofbox

\subsection{Proof of Corollary~\ref{corollary-strong-H1-conv-density-1}}
{Following the proof of Theorem~\ref{theorem-convergence-moreau-yosida} above, we consider a subsequence (to which we pass without change of notation) such that the $(u_{\lambda_j},m_{\lambda_j})$ converge to a solution $(u,m)$ of~\eqref{weakform} in the sense of~\eqref{m-conv-2-moreau-yosida} and that $\tilde{b}_j\rightharpoonup^* \tilde{b}_*$ in $L^\infty(\Omega;\R^\dim)$ as $j\to \infty$, where $\tilde{b}_j\coloneqq \frac{\partial H_{\lambda_j}}{\partial p}[\nabla u_{\lambda_j}]$ and $\tilde{b}_* \in D_pH[u]$.}
{Also, by hypothesis, the sequence $\{\tilde{b}_j\}_{j\in\mathbb{N}}$ is pre-compact in $L^1(\Omega;\mathbb{R}^d)${, so} after possibly passing to a further subsequence, we may assume without loss of generality that $\tilde{b}_j\to \tilde{b}_*$ in $L^1(\Omega;\R^\dim)$ as $j\to \infty$.}
Since $|\Omega|_d<\infty$, we deduce from the dominated convergence theorem that $\tilde{b}_j\to\tilde{b}_*$ in $L^s(\Omega;\mathbb{R}^d)$ for any $s\in [1,\infty)$ as $j\to\infty$. 
Let $r\in (2,2^*)$ be given, where $2^*\coloneqq \infty$ if $d\in\{1,2\}$ and $2^*\coloneqq \frac{2d}{d-2}$ if $d\geq 3$. 
{The Sobolev embedding theorem shows} that $m\in L^r(\Omega)$. Therefore, if we set $s\coloneqq \frac{2r}{r-2}$, we deduce from the triangle inequality and H\"older's inequality that 
\begin{multline}\label{mb-conv-est}
	\lim_{j\to\infty}\norm{m_{\lambda_j}\tilde{b}_j-m\tilde{b}_*}_{L^2(\Omega;\R^\dim)}
	\\ \leq \lim_{j\to\infty}\sqrt{2}\left(L_H\|m_{\lambda_{j}}-m\|_{L^2(\Omega)}+\|m\|_{L^r(\Omega)}\norm{\tilde{b}_j-\tilde{b}_*}_{L^{s}(\Omega;\mathbb{R}^d)}\right) 
	=0.
\end{multline}
This shows that $m_{\lambda_{j}}\tilde{b}_j\to m\tilde{b}_*$ in $L^2(\Omega;\mathbb{R}^d)$ along a subsequence as $j\to\infty$.
{Therefore, after testing~\eqref{weakform2-space-time-moreau-yosida} with $m_{\lambda_j}$ and passing to the limit, we obtain}
\begin{equation}
	\begin{aligned}
		\lim_{j\to\infty}\|\nabla m_{\lambda_j}\|_{L^2(\Omega;\mathbb{R}^d)}^2  &= \nu^{-1}\lim_{j\to\infty}\left(\langle G, m_{\lambda_j}\rangle_{H^{-1}\times H_0^1} - \int_{\Omega}m_{\lambda_j}\tilde{b}_j\cdot\nabla m_{\lambda_j}\mathrm{d}x\right)
		\\ & = \nu^{-1}\left(\langle G, m\rangle_{H^{-1}\times H_0^1} -\int_{\Omega}m\tilde{b}_*\cdot\nabla m\mathrm{d}x\right)=\|\nabla m\|_{L^2(\Omega;\mathbb{R}^d)}^2,
	\end{aligned}
\end{equation}
where the last equality follows from~\eqref{weakform2} tested with $m$.
{This shows that $\nabla m_{\lambda_j}\to \nabla m$ in $L^2(\Omega;\R^\dim)$ and thus $m_{\lambda_j}\to m$ in $H^1_0(\Omega)$ as $j\to\infty$.}
\hfill\proofbox

\section{Proof of Theorem~\ref{theorem-rate-of-convergence-moreau-yosida}}\label{sec-pfs-of-main-result-2}
\subsection{Preliminary results}

{To begin, we introduce the pointwise maximising set corresponding to the Hamiltonian~\eqref{Hamiltonian}. Define the set-valued map $\Lambda\colon\Omega\times\mathbb{R}^d\rightrightarrows \mathcal{A}$ via
	\begin{equation}
		\Lambda(x,p)\coloneqq\text{argmax}_{\alpha\in\mathcal{A}}\{b(x,\alpha)\cdot p-f(x,\alpha)\}.
	\end{equation}
	Following \cite{osborne2022analysis}, we associate with given $v\in W^{1,1}(\Omega)$ the set $\Lambda[v]$ of Lebesgue measurable functions $\alpha^*:\Omega\to\mathcal{A}$ that satisfy $\alpha^*(x)\in\Lambda(x,\nabla v(x))$ for a.e.\ $x\in\Omega.$ We will refer to each element of $\Lambda[v]$ as \emph{a measurable selection of }$\Lambda(\cdot,\nabla v(\cdot))$. 
	It is known that~$\Lambda[v]$ is non-empty for all $v\in W^{1,1}(\Omega)$, a result which ultimately rests upon the Kuratowski--Ryll-Nardzewski theorem~\cite{kuratowski1965general} on measurable selections, see also~\cite[Appendix~B]{smears2015thesis} for a detailed proof.}

{We will make use of the following semismoothness result for the Hamiltonians, which was first shown in~\cite[Theorem~13]{smears2014discontinuous} already for the more general case of fully nonlinear second~order HJB operators.}
\begin{lemma}\label{lemma-semismooth-tech-result}
	Let $v\in H^1(\Omega)$ be given. For each $\epsilon>0$, there exists a $R>0$, depending only on $v$, $\epsilon$, $\Omega$, and $H$, such that 
	\begin{equation}\label{semismooth-bound-general}
		\sup_{\alpha\in {\Lambda[w]}}\left\|H[\nabla w] - H[\nabla v] - b(\cdot,\alpha)\cdot \nabla (w-v)\right\|_{H^{-1}(\Omega)} \leq \epsilon\|v - w\|_{H^1(\Omega)},
	\end{equation}
	whenever $w\in H^1(\Omega)$ satisfies $\|v- w\|_{H^1(\Omega)}\leq R$.
\end{lemma}
\begin{proof}
	{Following the analysis of~\cite[Theorem~13]{smears2014discontinuous} adapted to the current context, it is known that, for every $s\in [1,2)$ and for every $v\in H^1(\Omega)$, we have 
		\begin{equation}\label{eq:semismoothness_lebesgue_spaces}
			\lim_{\norm{e}_{H^1(\Omega)}\tends 0}\sup_{\alpha\in \Lambda[v+e]} \frac{1}{\norm{e}_{H^1(\Omega)}}\norm{H[\nabla (v+e)]-H[\nabla v]-b(\cdot,\alpha)\cdot \nabla e}_{L^s(\Omega)}=0.
		\end{equation}
		The proof of~\eqref{eq:semismoothness_lebesgue_spaces} follows the same ideas as in~\cite[Theorem~13]{smears2014discontinuous}, with the difference that \cite{smears2014discontinuous} treats the case of fully nonlinear second-order HJB operators in spaces of functions with piecewise second-order Sobolev regularity; whereas~\eqref{eq:semismoothness_lebesgue_spaces} considers the first order Sobolev space $H^1$ since the Hamiltonian here depends only on the gradient of the functions and not on the second derivatives.
		Note also that the proof of~\eqref{eq:semismoothness_lebesgue_spaces} uses mainly the boundedness of $\Omega$, the compactness of the control set $\mathcal{A}$, and the continuity of the data $b$ and $f$.
		We then obtain~\eqref{semismooth-bound-general} by taking $w=v+e$ and noting that there exists some $s\in[1,2)$ such that $L^s(\Omega)$ is continously embedded in $H^{-1}(\Omega)$ by the Sobolev embedding theorem.}
\end{proof}

\begin{remark}
	We emphasize that the semismoothness of $H$  shown in Lemma~\ref{lemma-semismooth-tech-result} is strictly weaker than differentiability of $H$. The key difference is that in~\eqref{semismooth-bound-general}, we have $\alpha\in \Lambda[w]$, instead of $\Lambda[v]$. In fact, the simple example $H(x,p)=\sup_{\alpha\in \mathcal{A}}\{\alpha\cdot p\}=|p|$, where $\mathcal{A}=\overline{B_1(0)}\subset\mathbb{R}^d$, shows that we cannot replace $\Lambda[w]$ by $\Lambda[v]$ in~\eqref{semismooth-bound-general} in general.
\end{remark}

{The semismoothness of the Hamiltonian enables us to show the following bound between the value functions of the original and regularized problems.}
\begin{lemma}\label{lemma-u-lam-approx-rate}
	{Assume the hypotheses~\ref{H:G-postive}, \ref{H:F-strong-mono-2}, \ref{H:F-Lipschitz-continuous}, and~\ref{H:reg-family}.
		Let $(u,m)$ and $(u_{\lambda},m_{\lambda})$ be the respective unique solutions of~\eqref{weakform} and~\eqref{weakform-moreau-yosida}.
		Then, for all $\lambda$ sufficiently small, we have
		\begin{equation}\label{hjb-reg-approx-error}
			\|u-u_{\lambda}\|_{H^1(\Omega)}\lesssim \|m-m_{\lambda}\|_{\mathcal{X}} + \omega(\lambda).
		\end{equation}
		The hidden constant in~\eqref{hjb-reg-approx-error} depends only on 
		$\Omega$, $\nu$, $d$, $L_H$, and $L_F$.}
\end{lemma}

\begin{proof}
	{For each $\lambda\in (0,1]$, choose an arbitrary $\alpha_{\lambda}\in \Lambda[u_\lambda]$, and define the operator $L_\lambda:H^1(\Omega)\to H^{-1}(\Omega)$ by 
		\begin{equation}
			\langle L_{\lambda}w,v\rangle_{H^{-1}\times H_0^1}\coloneqq \int_{\Omega}\nu\nabla w\cdot\nabla v + b(x,\alpha_\lambda)\cdot\nabla wv\mathrm{d}x\quad \forall w,v\in H^1(\Omega).
		\end{equation}
		Recalling the definition of the set of operators $\mathcal{G}(L_H)$, it is clear that $L_{\lambda}\in\mathcal{G}(L_H)$ since it is of the form~\eqref{eq:wmp_1} and $\norm{b(\cdot,\alpha_\lambda)}_{L^\infty(\Omega;\R^\dim)}\leq \norm{b}_{C(\overline{\Omega}\times\mathcal{A};\R^\dim)}\leq L_H$.
		We then obtain from~\eqref{weakform1} and from~\eqref{weakform1-space-time-moreau-yosida} that
		\begin{multline}
			\langle L_\lambda(u_\lambda-u),\psi\rangle_{H^{-1}\times H_0^1}
			=\langle F[m_\lambda] - F[m],\psi\rangle_{H^{-1}\times H_0^1} + \int_{\Omega}(H[\nabla u_{\lambda}] - H_\lambda[\nabla u_{\lambda}])\psi\mathrm{d}x
			\\
			- \int_{\Omega}(H[\nabla u_{\lambda}] - H[\nabla u] - b(\cdot,\alpha_{\lambda})\cdot \nabla (u_{\lambda}-u))\psi\mathrm{d}x
		\end{multline}
		Using Lemma~\ref{lemma-uniform-L-inv-bound}, we deduce that 
		\begin{multline}\label{eq:u_lam_rate_1}
			\|u-u_{\lambda}\|_{H^1(\Omega)}\leq C_* \left(\|F[m] - F[m_{\lambda}]\|_{H^{-1}(\Omega)} + \|H_{\lambda}[\nabla u_{\lambda}] - H[\nabla u_{\lambda}]\|_{L^2(\Omega)} 
			\right.\\
			\left.\quad\quad\quad+ \|H[\nabla u] - H[\nabla u_{\lambda}] - b(\cdot,\alpha_{\lambda})\cdot \nabla (u-u_{\lambda})\|_{H^{-1}(\Omega)}\right),
		\end{multline} 
		where $C_*> 0$ depends on only $\Omega$, $\nu$, $d$ and  $L_H$.
		Theorem~\ref{theorem-convergence-moreau-yosida}, in particular the strong convergence $u_\lambda \to u$ in $H^1_0(\Omega)$ as $\lambda\to 0$, and Lemma~\ref{lemma-semismooth-tech-result} together imply that, for all $\lambda$ sufficiently small,
		\begin{equation}\label{eq:u_lam_rate_2}
			\|H[\nabla u_\lambda] - H[\nabla u] - b(\cdot,\alpha_\lambda)\cdot \nabla (u_\lambda-u)\|_{H^{-1}(\Omega)}\leq \frac{1}{2C_*}\|u-u_{\lambda}\|_{H^1(\Omega)},
		\end{equation}
		where $C_*$ is the constant from~\eqref{eq:u_lam_rate_1}.
		After combining~\eqref{eq:u_lam_rate_1} with the hypothesis~\ref{H:F-Lipschitz-continuous} and with~\eqref{H-bounds:reg-unif-approx_1}, and with~\eqref{eq:u_lam_rate_2}, we see that
		\begin{equation} 
			\begin{split}
				\|u-u_{\lambda}\|_{H^1(\Omega)}&\lesssim\|F[m] - F[m_{\lambda}]\|_{H^{-1}(\Omega)} + \norm{H[\nabla u_\lambda]-H_\lambda[\nabla u_\lambda]}_{L^2(\Omega)} \\
				&\lesssim \norm{m-m_{\lambda}}_{\mathcal{X}}+\omega(\lambda),
			\end{split}
		\end{equation}
		for all $\lambda$ sufficiently small.
		The hidden constant in the above bounds depends only on $\Omega$, $\nu$, $d$, and $L_H$ and $L_F$.
		This completes the proof of~\eqref{hjb-reg-approx-error}.}
\end{proof}

\begin{lemma}\label{lemma-m-lam-approx-rate}
	Assume the hypotheses~\ref{H:G-postive}, \ref{H:F-strong-mono-2}, \ref{H:F-Lipschitz-continuous}, and~\ref{H:reg-family}.
	{Let $(u,m)$ and $(u_{\lambda},m_{\lambda})$, $\lambda\in(0,1]$, be the respective unique solutions of~\eqref{weakform} and~\eqref{weakform-moreau-yosida}.}
	Then, for all $\lambda \in (0,1]$,
	\begin{equation}
		\|m - m_{\lambda}\|_{\mathcal{X}}\lesssim \|G\|_{H^{-1}(\Omega)}^{\frac{1}{2}}\omega(\lambda)^{\frac{1}{2}}.
	\end{equation}
	where the hidden constant depends only on $c_F$, $d$, $\nu$, $L_H$ and $\Omega$.
\end{lemma}
\begin{proof}
	Fix $\lambda\in (0,1]$. Test both~\eqref{weakform1} and~\eqref{weakform1-space-time-moreau-yosida} with $\psi = m - m_{\lambda}$ {and} subtract the resulting equations to obtain
	\begin{equation}\label{u-eqn-diff}
		\begin{split}
			&\langle F[m] - F[m_{\lambda}],m - m_{\lambda}\rangle_{H^{-1}\times H_0^1}
			\\&=\int_{\Omega}\nu\nabla (u- u_{\lambda})\cdot \nabla (m - m_{\lambda})+m(H[\nabla u] - H_{\lambda}[\nabla u_{\lambda}])+m_{\lambda}(H_{\lambda}[\nabla u_{\lambda}] - H[\nabla u])\mathrm{d}x. 
		\end{split}
	\end{equation} 
	Then, test both~\eqref{weakform2} and~\eqref{weakform2-space-time-moreau-yosida} with $\phi = u - u_{\lambda}$, and subtract the resulting equations to get 
	\begin{equation}\label{m-eqn-diff}
		\begin{split}
			&\int_{\Omega}\nu\nabla (u- u_{\lambda})\cdot\nabla (m - m_{\lambda}) 
			+mb_*\cdot\nabla (u - u_{\lambda}) + m_{\lambda}\frac{\partial H_{\lambda}}{\partial p}[\nabla u_{\lambda}]\cdot\nabla (u_{\lambda} - u)\mathrm{d}x= 0. 
		\end{split}
	\end{equation} 
	We then subtract~\eqref{m-eqn-diff} from~\eqref{u-eqn-diff} to thus obtain
	\begin{equation}
		\begin{split}
			&\langle F[m] - F[m_{\lambda}],m - m_{\lambda}\rangle_{H^{-1}\times H_0^1}
			\\
			&\quad\quad\quad\quad\quad\quad=\int_{\Omega}m_{\lambda}\left(H_{\lambda}[\nabla u_{\lambda}] - H[\nabla u]  + \frac{\partial H_{\lambda}}{\partial p}[\nabla u_{\lambda}]\cdot\nabla (u - u_{\lambda})\right)\mathrm{d}x 
			\\
			&\quad\quad\quad\quad\quad\quad\quad\quad+\int_{\Omega} m\left(H[\nabla u] - H_{\lambda}[\nabla u_{\lambda}] + b_*\cdot\nabla (u_{\lambda} - u)\right) \mathrm{d}x.
		\end{split}
	\end{equation} 
	Notice that the hypothesis~\ref{H:G-postive} on $G$, together with the Weak Maximum Principle and the Comparison Principle, implies that $m$ and $m_{\lambda}$ are both nonnegative a.e.\ in $\Omega$. This fact, together with the convexity of $H_{\lambda}$ w.r.t.\ $p$ and the definition of the inclusion $\tilde{b}_*\in D_pH[u]$, we deduce that
	\begin{equation*}
			\int_{\Omega}m_{\lambda}\left(H_{\lambda}[\nabla u_{\lambda}] - H[\nabla u]  + \frac{\partial H_{\lambda}}{\partial p}[\nabla u_{\lambda}]\cdot\nabla (u - u_{\lambda})\right)\mathrm{d}x \leq  \int_{\Omega}m_{\lambda}(H_{\lambda}[\nabla u]- H[\nabla u] )\mathrm{d}x,
	\end{equation*}
	and 
	$$\int_{\Omega} m\left(H[\nabla u] - H_{\lambda}[\nabla u_{\lambda}] + b_*\cdot\nabla (u_{\lambda} - u)\right) \mathrm{d}x\leq \int_{\Omega}m(H[\nabla u_{\lambda}] - H_{\lambda}[\nabla u_{\lambda}])\mathrm{d}x.$$
	Consequently,
	\begin{equation}\label{eq:m-lam-approx_1}
		\begin{split}
			\langle F[m] - F[m_{\lambda}],m - m_{\lambda}\rangle_{H^{-1}\times H_0^1}
			\leq 
			&\int_{\Omega}m_{\lambda}(H_{\lambda}[\nabla u]- H[\nabla u] )\mathrm{d}x
			\\&\quad\quad\quad\quad+\int_{\Omega}m(H[\nabla u_{\lambda}] - H_{\lambda}[\nabla u_{\lambda}])\mathrm{d}x.
		\end{split}
	\end{equation}
	{We then combine~\eqref{eq:m-lam-approx_1} with the strong monotonicity of $F$~\ref{H:F-strong-mono-2}, the uniform bound $|H_\lambda-H|\leq \omega(\lambda)$ of~\eqref{H-bounds:reg-unif-approx_1}, and the Cauchy--Schwarz inequality to obtain
		\begin{equation}\label{eq:m-lam-approx_2}
			\norm{m-m_{\lambda}}_\mathcal{X}^2 \leq c_F^{-1}\langle F[m] - F[m_{\lambda}],m - m_{\lambda}\rangle_{H^{-1}\times H_0^1} \lesssim
			\left(\|m_{\lambda}\|_{L^2(\Omega)}+\|m\|_{L^2(\Omega)}\right)\omega(\lambda),
		\end{equation}
		where the hidden constant depends only on $c_F$ and $\Omega$.}
	Using Lemma~\ref{lemma-uniform-L-inv-bound}, we deduce from the KFP equations satisfied by $m$ and $m_{\lambda}$, respectively, that $ \|m\|_{L^2(\Omega)}\lesssim\|G\|_{H^{-1}(\Omega)}$ and $ \|m_{\lambda}\|_{L^2(\Omega)}\lesssim\|G\|_{H^{-1}(\Omega)}$. Therefore, 
	\begin{equation}
		\|m - m_{\lambda}\|_{\mathcal{X}}^{2}{\lesssim \|G\|_{H^{-1}(\Omega)}\omega(\lambda),}
	\end{equation}
	{where the constant depends only on $\Omega$, $\dim$, $\nu$, $c_F$, and $L_H$.}
	This completes the proof as $\lambda\in (0,1]$ was arbitrary.
\end{proof}

The proof of Theorem~\ref{theorem-rate-of-convergence-moreau-yosida} is now straightforward.
\begin{proof}[Proof of Theorem~\ref{theorem-rate-of-convergence-moreau-yosida}]
	{The result is immediate by combining~the conclusions of Lemmas~\ref{lemma-u-lam-approx-rate} and~\ref{lemma-m-lam-approx-rate}, i.e.
		\begin{equation}\label{eq:proof_m_lam_rate}
			\norm{u-u_\lambda}_{H^1(\Omega)}+\norm{m-m_\lambda}_{\mathcal{X}}\lesssim \norm{m-m_\lambda}_{\mathcal{X}}+\omega(\lambda)\lesssim \omega(\lambda)^{\frac{1}{2}},
		\end{equation}
		where we have used the trivial bound $\omega(\lambda)\leq C_\omega \omega(\lambda)^{\frac{1}{2}}$ where $C_\omega\coloneqq \sup_{\sigma\in[0,1]}\omega(\sigma)^{\frac{1}{2}}$.
		Note that the constant in~\eqref{eq:proof_m_lam_rate} depends only on $\Omega$, $\dim$, $\nu$, $L_H$, $L_F$, $c_F$ and also $\sup_{\sigma\in[0,1]}\omega(\lambda)$.}
\end{proof}

\section{Examples}\label{sec-examples}
{In this section, we show the sharpness of the conclusions of the analysis above with respect to several aspects.
	Recall that Theorem~\ref{theorem-convergence-moreau-yosida} shows the weak convergence in $H^1$ of subsequences of the densities of the regularized problems to a density of the MFG PDI.
	Also, Corollary~\ref{corollary-strong-H1-conv-density-1} gives strong convergence in $H^1$ under an additional hypothesis.
	To show that such an additional hypothesis cannot be removed in general, in Section~\ref{sec-eg-4} below we give an example where the density function approximations given by the regularized problems do not converge strongly in the $H^1$-norm.
	Theorem~\ref{theorem-convergence-moreau-yosida} is sharp in this regard.
	
	In the second example, given in Section~\ref{sec-eg-3}, we consider a situation where each regularized problem has a unique solution pair $(u_{\lambda_j},m_{\lambda_j})$, yet by taking different subsequences, we arrive at different solutions of the MFG PDI. 
	Thus the convergence along subsequences shown in Theorem~\ref{theorem-convergence-moreau-yosida} cannot be improved to convergence of the whole sequence, even in cases where each regularized problem has a unique solution.}

\subsection{Density function approximations may not converge strongly in the $H^1$-norm}\label{sec-eg-4}
{We now present an example where the densities of the regularized problems do not converge strongly in the $H^1$-norm.}  This shows that, {for the general class of regularized Hamiltonians satisying \ref{H:reg-family}}, the weak convergence in the $H^1$-norm of  $\{m_{\lambda_j}\}_{j\in\mathbb{N}}$ given in Theorem~\ref{theorem-convergence-moreau-yosida} is sharp {in general.}
The {key features of the} example are summarized in the following proposition.
\begin{proposition}\label{prop-eg-3}
	Let $\Omega {=} (0,1)\subset\mathbb{R}$, and let $H:\overline{\Omega}\times\mathbb{R}\to \mathbb{R}$ be given by {$H(x,p)= \sup_{\alpha\in [-1,1]}\{x\alpha p\}= x|p|$}, for all $(x,p)\in \overline{\Omega}\times\mathbb{R}$. Set $F\equiv 0$ on $L^2(\Omega)$ and let $G\equiv 1\in L^2(\Omega)\subset H^{-1}(\Omega)$. 
	There exists a $m\in H^1_0(\Omega)$ and a family of regularizations $\{H_{\lambda_j}\}_{j\in\N}$ satisfying~\ref{H:reg-family}, with $\lambda_j\to 0$ as $j\to\infty$, such that 
	\begin{itemize}
		\item $(0,m)\in H_0^1(\Omega)\times H_0^1(\Omega)$ {is a solution of}~\eqref{weakform} with $m\neq 0$ in $H_0^1(\Omega)$,
		\item $m_{\lambda_j}\rightharpoonup m $ in $H_0^1(\Omega)$ as $j\to\infty$, and
		\begin{equation}\label{eq:eg-4-noconv}
			\lim_{j\to\infty}\|m_{\lambda_j}-m\|_{H^1(\Omega)}>0.
		\end{equation}
	\end{itemize}
\end{proposition}

\begin{proof}
	In this setting one can show that the Moreau--Yosida regularization of $H$, for each $\lambda\in (0,1]$, is given by
	\begin{equation}\label{eq:eg-4-MY-reg}
		\mathcal{H}_{\lambda}(x,p)
		=
		\begin{cases}
			x|p|-\frac{\lambda x^2}{2} &\text{ if } |p|\geq \lambda x,
			\\
			\frac{|p|^2}{2\lambda} &\text{ if } |p|\leq \lambda x,
		\end{cases}\quad 
		\text{and}\quad  
		\frac{\partial \mathcal{H}_{\lambda}}{\partial p}(x,p) =
		\begin{cases}
			\text{sgn}(p)x &\text{ if } |p|\geq \lambda x,
			\\
			\frac{p}{\lambda} &\text{ if } |p|\leq \lambda x,
		\end{cases}
	\end{equation} for $(x,p)\in [0,1]\times\mathbb{R}$.
	{For each $\lambda\in (0,1]$, let $H_{\lambda}$ be defined by}
	\begin{equation}\label{eq:eg-4-H_lam_def}
		H_{\lambda}(x,p)\coloneqq \mathcal{H}_{\lambda}\left(x,p - x\cos\left({x}{\lambda^{-1}}\right)\lambda\right) - \frac{x^2\cos^2({x}{\lambda^{-1}})\lambda}{2}\quad\forall (x,p)\in [0,1]\times\mathbb{R}.
	\end{equation}
	{It is straightforward to use Lemma~\ref{lemma-moreau-yosida-approx-properties} to check that the assumption~\ref{H:reg-family} is satisfied by $\{H_{\lambda}\}_{\lambda\in (0,1]}$. 
		In particular, it is found that $\left\lvert H_\lambda(x,p)-H(x,p)\right\rvert \leq 2 \lambda$ for all $x\in [0,1]$ and all $p\in \R$.} {Note that $H_\lambda(x,0)=0$ for all $x\in [0,1]$.}
	{It also follows immediately from~\eqref{eq:eg-4-MY-reg} and~\eqref{eq:eg-4-H_lam_def} that}
	$$\frac{\partial H_{\lambda}}{\partial p}(x,0) = \frac{\partial \mathcal{H}_{\lambda}}{\partial p}\left(x,- x\cos\left({x}{\lambda^{-1}}\right)\lambda\right) = - x\cos\left({x}{\lambda^{-1}}\right)\quad\forall x\in (0,1).$$
	Now consider the sequence $\{\lambda_j\}_{j\in\mathbb{N}}$ given by $\lambda_j\coloneqq 1/j$, $j\in\mathbb{N}$.
	{Then, for each $j\in \N$,} the regularized problem~\eqref{weakform-moreau-yosida} admits a unique solution $(u_{\lambda_j},m_{\lambda_j})\in H_0^1(\Omega)\times H_0^1(\Omega)$ where $u_{\lambda_j} = 0 $ a.e.\ in $\Omega$ and $m_{\lambda_j}$ is {the unique solution of}
	\begin{equation}\label{prop-3-m_j-eqn-pf}
		\int_{\Omega}\nu \partial_xm_{\lambda_j}\partial_x\phi + m_{\lambda_j}b_j\partial_x\phi\mathrm{d}x = \int_{\Omega}\phi\mathrm{d}x\quad\forall \phi\in H_0^1(\Omega),
	\end{equation} 
	where $b_{j}\coloneqq \frac{\partial H_{\lambda_j}}{\partial p}[\partial_x u_{\lambda_j}] = \frac{\partial H_{\lambda_j}}{\partial p}(x,0)= - x\cos\left(jx\right)$ for all $x\in (0,1)$. 
	{The solution $m_{\lambda_j}$ can be computed explicitly, in particular}
	\begin{equation}
		m_{\lambda_j}(x) =\frac{1}{\nu\gamma_j(x)\int_{\Omega}\gamma_j(s)\mathrm{d}s}\left(\int_{\Omega}s\gamma_j(s)\mathrm{d}s\int_0^x\gamma_j(s)\mathrm{d}s - \int_0^xs\gamma_j(s)\mathrm{d}s\int_{\Omega}\gamma_j(s)\mathrm{d}s\right),
	\end{equation}
	for $x\in [0,1]$, where $\gamma_j(s)\coloneqq \exp\left(-(j\nu)^{-1}x\sin(jx) - (j^2\nu)^{-1}\cos(jx)\right)$ for $s\in [0,1]$. 
	
	Lemma~\ref{lemma-uniform-L-inv-bound} shows that the sequence $\{m_{\lambda_j}\}_{j\in\mathbb{N}}$ is uniformly bounded in $H_0^1(\Omega)$.
	{The Riemann--Lebesgue lemma shows that} the entire sequence $\{m_{\lambda_j}\}_{j\in\mathbb{N}}$ converges weakly to $m\in H_0^1(\Omega)$ which is the unique solution of
	\begin{equation}\label{prop-3-m-eqn-pf}
		\int_{\Omega}\nu \partial_xm\partial_x\phi\mathrm{d}x=\int_{\Omega}\phi\mathrm{d}x\quad\forall\phi\in H_0^1(\Omega),
	\end{equation}
	with $m$ given explicitly by
	\begin{equation}
		m(x)=\frac{1}{2\nu}x(1-x)\quad \forall x\in [0,1].
	\end{equation}
	{Furthermore,} we also have $m_{\lambda_j}\to m$ in $L^{\infty}(\Omega)$ {as $j\to\infty$.}

	We now {prove} that $\{m_{\lambda_j}\}_{j\in\mathbb{N}}$ does not converge strongly to $m$ in the $H^1$-norm {by direct calculation}. Notice that, for each $j\in\mathbb{N}$, \eqref{prop-3-m_j-eqn-pf} implies that 
	\begin{equation}
		\partial_xm_{\lambda_j} = \nu^{-1}x\cos(jx)m_{\lambda_j} - \nu^{-1}(x+c_j)
	\end{equation}
	where $c_j\coloneqq  - \left(\int_{\Omega}\gamma_j(s)\mathrm{d}s\right)^{-1}\int_{\Omega}s\gamma_j(s)\mathrm{d}s \in \R$.
	{It is clear that} $\gamma_j$ converges uniformly to the constant function $1$ on $[0,1]$ as $j\to\infty$. 
	Therefore $c_j\to -\frac{1}{2}$ as $j\to\infty$, and 
	\begin{equation}
		\partial_x(m_{\lambda_j}-m)(x)=\nu^{-1}x\cos(jx)m_{\lambda_j}(x) - \nu^{-1}\left(c_j+\frac{1}{2}\right)\quad\forall x\in [0,1].
	\end{equation} We then get
	\begin{equation}
		\begin{split}
			&\int_{\Omega}|\partial_x(m_{\lambda_j}-m)|^2\mathrm{d}x\\
			&=\nu^{-2}\int_{\Omega}x^2\cos^2(jx)m_{\lambda_j}^2\mathrm{d}x -2\nu^{-2}\left(c_j+\frac{1}{2}\right)\int_{\Omega}x\cos(jx)m_{\lambda_j}\mathrm{d}x +\nu^{-2}\left(c_j+\frac{1}{2}\right)^2
			\\
			&=\frac{1}{2\nu^2}\int_{\Omega}x^2m_{\lambda_j}^2\mathrm{d}x +  \frac{1}{2\nu^2}\int_{\Omega}x^2\cos(2jx)(m_{\lambda_j}^2-m^2)\mathrm{d}x +  \frac{1}{2\nu^2}\int_{\Omega}x^2\cos(2jx)m^2\mathrm{d}x \\
			&\quad\quad-2\nu^{-2}\left(c_j+\frac{1}{2}\right)\int_{\Omega}x\cos(jx)m_{\lambda_j}\mathrm{d}x +\nu^{-2}\left(c_j+\frac{1}{2}\right)^2.
		\end{split}
	\end{equation}
	The convergences $m_{\lambda_j}\to m$ in $L^{\infty}(\Omega)$ and $c_j\to -1/2$ as $j\to\infty$, together with the Riemann-Lebesgue Lemma, imply that
	\begin{equation}\label{mj-no-conv}
		\lim_{j\to\infty}\int_{\Omega}|\partial_x(m_{\lambda_j}-m)|^2\mathrm{d}x = \int_{\Omega}\frac{x^2}{2\nu^2}m^2\mathrm{d}x >0 .
	\end{equation}
	{This shows~\eqref{eq:eg-4-noconv}, i.e.\ the sequence $m_{\lambda_j}$ does not converge strongly to $m$ in the $H^1$-norm. Since the whole sequence is weakly convergent to $m$ in $H^1_0$, we conclude that there is no strongly convergent subsequence with respect to the $H^1_0$-norm.}
\end{proof}

{
	\begin{remark}
		Considering the proof above, it is easy to show that the whole sequence $\frac{\partial H_{\lambda_j}}{\partial p}[\partial_x u_{\lambda_j}]=-x\cos(jx)$ does not satisfy the pre-compactness hypothesis of Corollary~\ref{corollary-strong-H1-conv-density-1}.
	\end{remark}
}
{The conclusion here is that Theorem~\ref{theorem-convergence-moreau-yosida} is sharp in general with regards to the the weak convergence of the densities in $H^1$, and that strong convergence is only possible under some additional hypotheses, such as in Corollary~\ref{corollary-strong-H1-conv-density-1}.}

\subsection{A sequence of regularized problems with two subsequences that converge strongly to different solutions}\label{sec-eg-3}

{We now give an example where the each regularized problem has a unique solution, yet different subsequences of the solutions of the regularized problems converge to different solutions of the MFG PDI.
	This shows that in Theorem~\ref{theorem-convergence-moreau-yosida}, we generally cannot expect convergence for the whole sequence to a unique limit, even when each regularized problem has a unique solution.}

\begin{proposition}\label{prop-eg-2}
	Let $\Omega\subset \mathbb{R}^d$, $d\in \mathbb{N}$, denote a bounded domain with Lipschitz boundary, and let $H:\overline{\Omega}\times\mathbb{R}^d\to \mathbb{R}$ be given by $H(x,p) = \sup_{\alpha\in \overline{B_1(0)}}\{\alpha \cdot p\} = |p|$ {for all $(x,p)\in\overline{\Omega}\times \R^\dim$.}
	Set $F\equiv 0$ on $L^2(\Omega)$ and let $G\equiv 1\in L^2(\Omega)\subset H^{-1}(\Omega)$. 
	There exists a family $\{H_{\lambda_j}\}_{j\in\N}$ satisfying~\ref{H:reg-family}, with $\lambda_j\to 0$ as $j\to\infty$, and two distinct $m_1,\,m_2\in H^1_0(\Omega)$, $m_1\neq m_2$, such that
	\begin{itemize}
		\item the pairs $(0,m_1),\,(0,m_2) \in H_0^1(\Omega)\times H_0^1(\Omega)$ are both {solutions of~\eqref{weakform}}, 
		\item for each $j\in\mathbb{N}$, the regularized problem~\eqref{weakform-moreau-yosida} with $\lambda=\lambda_j$ has a unique solution $(u_{\lambda_j},m_{\lambda_j})\in H^1_0(\Omega)\times H^1_0(\Omega)$,
		\item {We have}
		\begin{equation}\label{eq:eg-3-m_alternating}
			\begin{aligned}
				m_{\lambda_{2k+1}} = m_1, && m_{\lambda_{2k}} = m_2, &&\forall k\in \N.
			\end{aligned}
		\end{equation}
	\end{itemize}
\end{proposition}

{Note that it is a trivial consequence of~\eqref{eq:eg-3-m_alternating} that $m_{\lambda_{2k+1}}\to m_1$ and $m_{\lambda_{2k}}\to m_2$ as $k\to \infty$, in any norm.}

\begin{proof}
	{It is clear that the unique solution $u$ of the HJB equation~\eqref{weakform1} is simply $u=0$ in $\Omega$ in the case where $F\equiv 0$ and $H(x,p)=|p|$.}
	{Hence,} the set of solutions to the weak MFG PDI~\eqref{weakform} is {a non-singleton set that consists} of all pairs $(u,m)$ where $u=0$ a.e.\ in $\Omega$ and $m\in H_0^1(\Omega)$ solves the KFP equation~\eqref{weakform2} for some $\tilde{b}_*\in D_pH[u]=\{b\in L^\infty(\Omega;\R^\dim)\colon \norm{b}_{L^\infty(\Omega;\R^\dim)}\leq L_H\}$.
	{Let $m_1$ and $m_2$ denote the corresponding densities for the constant vector fields $b_1=(1,0,\dots,0)\in D_pH[u]$ and $b_2=(-1,0,\dots,0)\in D_pH[u]$.}
	
	Let $\mathcal{H}_{\lambda}$, $\lambda\in (0,1]$, denote the Moreau--Yosida regularization of $H(x,p) = |p|$, which is given by 
	\begin{equation} 
		\mathcal{H}_{\lambda}(x,p)
		=
		\begin{cases}
			|p|-\frac{\lambda}{2} &\quad \text{if } |p|\geq \lambda,
			\\
			\frac{|p|^2}{2\lambda} &\quad \text{if } |p|\leq \lambda,
		\end{cases}\quad 
		\text{and}\quad  
		\frac{\partial \mathcal{H}_{\lambda}}{\partial p}(p) =
		\begin{cases}
			|p|^{-1}p &\quad \text{if } |p|\geq \lambda,
			\\
			\frac{p}{\lambda} &\quad \text{if } |p|\leq \lambda.
		\end{cases}
	\end{equation} 
	Consider the family of regularized Hamiltonians $\{H_{\lambda}\}_{0<\lambda\leq 1}$ defined by 
	\begin{equation}
		H_{\lambda}(x,p)\coloneqq \mathcal{H}_{\lambda}(p - {q}_{\lambda}) - \frac{\cos^2(\lambda^{-1})\lambda}{2}\quad \forall (x,p)\in\overline{\Omega}\times \mathbb{R}^d,
	\end{equation}
	where ${q}_{\lambda}\coloneqq (\cos(\lambda^{-1})\lambda,0,0,\cdots,0)\in \R^\dim$ for each $\lambda\in (0,1]$. 
	{
		It is straightforward to use Lemma~\ref{lemma-moreau-yosida-approx-properties} to check that the assumption~\ref{H:reg-family} is satisfied by $\{H_{\lambda}\}_{\lambda\in (0,1]}$.
		Indeed, the Lipschitz continuity, convexity, and continuous differentiability of $H_\lambda$ with respect to $p$ all follow immediately from Lemma~\ref{lemma-moreau-yosida-approx-properties}.
		Furthermore, the triangle inequality and Lipschitz continuity~\eqref{bounds:lipschitz}, noting that $L_H=1$ in this example, imply that, for any $x\in\overline{\Omega}$ and any $p\in \R^\dim$,
		\begin{equation}
			|H_\lambda(x,p)-H(x,p)|\leq \left\lvert\mathcal{H}_{\lambda}(x,p-q_\lambda)-H(x,p-q_\lambda)\right\rvert + \lvert q_\lambda \rvert + \frac{\lambda}{2} \leq 2 \lambda,
		\end{equation}
		where we have used the bound~\eqref{moreau-yosida-bounds:reg-unif-approx_1} and $|q_\lambda|\leq \lambda$.
	}Observe that $H_{\lambda}(x,0) = 0$ for all $\lambda\in (0,1)$.  
	Now, take the sequence $\{\lambda_j\}_{j\in\mathbb{N}}$ given by $\lambda_j\coloneqq 1/(\pi j)$, $j\in\mathbb{N}$. 
	Then, by considering the regularized problem~\eqref{weakform-moreau-yosida} with $\lambda = \lambda_j$ for $j\in\mathbb{N}$, it is clear that the unique solution of the regularized HJB equation~\eqref{weakform1-space-time-moreau-yosida} is $u_{\lambda_j} =0$ in $\Omega$.
	Hence $m_{\lambda_j}$ is the unique solution of~\eqref{weakform2-space-time-moreau-yosida} where the advective vector field $\frac{\partial {H}_{\lambda_j}}{\partial p}[\nabla {u}_{\lambda_j}]$ is given by
	\begin{equation}
		\frac{\partial H_{\lambda_j}}{\partial p}[\nabla u_{\lambda_j}]=
		\frac{\partial \mathcal{H}_{\lambda_j}}{\partial p}(-q_{\lambda_j}) = ((-1)^{j+1},0,0,...,0)\in \R^\dim, \quad\forall j\in\N.
	\end{equation}
	{It is then clear that $ \frac{\partial H_{\lambda_j}}{\partial p}[\nabla u_{\lambda_j}]$ equals $b_1$ for all odd $j$ and equals $b_2$ for all even $j$.}
	{This implies~\eqref{eq:eg-3-m_alternating}.}
	
	To conclude, we now show that $m_1\neq m_2$ in $H_0^1(\Omega)$. Suppose for contradiction that $m_1=m_2=:m$ in $H_0^1(\Omega)$. We then obtain that $\int_{\Omega}m({v}_1-{v}_2)\cdot\nabla \phi\mathrm{d}x = 0\Longleftrightarrow \int_{\Omega}m\partial_{x_1}\phi\mathrm{d}x = 0$ for all $\phi\in H_0^1(\Omega)$. This then implies that $\partial_{x_1}m = 0$ a.e.\ in $\Omega$. But, since $\Omega$ is bounded and $m\in H^1_0(\Omega)$, the Poincar\'e inequality $\|m\|_{L^2(\Omega)}\leq C_{\Omega}\|\partial_{x_1}m\|_{L^2(\Omega)}$ {implies that} $m=0$ in $L^2(\Omega)$, which gives $m_1=m_2=m=0$ in $H_0^1(\Omega)$. But this contradicts~\eqref{weakform2} where the r.h.s.\ {is nonzero with $G\equiv 1 \in L^2(\Omega)\subset H^{-1}(\Omega)$}. Hence we see that $m_1\neq m_2$ in $H_0^1(\Omega)$.
\end{proof}

{
	\begin{remark}
		A further perspective on~Proposition~\ref{prop-eg-2} is that, in general, a single choice of regularizing sequence might not be sufficient to approximate all solutions of the PDI.
		Thus, in applications to problems with nonunique solutions, it may be necessary in some cases to consider multiple different regularizations to approximate different solutions of the PDI.
	\end{remark}
}

{
	\begin{remark}
		Let us relate this example to the discussion in Section~\ref{sec-introduction}.
		Observe that the case $F\equiv 0$ and $H(x,p) = \sup_{\alpha\in \overline{B_1(0)}}\{\alpha \cdot p\} =|p|$ (where $b\equiv \alpha$ and $f\equiv 0$) corresponds to a model of a MFG where the underlying control problem of the players is quite degenerate, since the players' cost is independent of both the controls and the population distribution. 
		The players are thus entirely decoupled from one another and are free to choose among infinitely many optimal controls, so it is natural that there should be infinitely many solutions of the PDI.
		This is in stark contrast with the regularized problems, where the differentiability of the regularized Hamiltonian implies the uniqueness of the optimal drift, compare with~\eqref{intro-convv-formula}, and thus fixes the players' dynamics.
		This gives a concrete illustration of some essential differences in the structure of the Nash equilibria between MFG with differentiable Hamiltonians and those with nondifferentiable Hamiltonians.
	\end{remark}
}

\bibliographystyle{siamplain_NoURL}
\bibliography{p4-references-revision}

\end{document}